\documentclass[reqno]{amsart}

\usepackage{amsmath,amsfonts,amssymb,amscd,verbatim,delarray,fullpage}

\usepackage{stmaryrd}
\usepackage{amsthm}
\usepackage{url}
\usepackage[polutonikogreek,english]{babel}

\DeclareMathOperator{\num}{Num}
\DeclareMathOperator{\den}{Den}
\DeclareMathOperator{\sgn}{sgn}

\DeclareMathOperator{\frob}{Frob}

\chardef\bslash=`\\ 

\begin{document}


\newtheorem{Theorem}{Theorem}[section]

\newtheorem{cor}[Theorem]{Corollary}

\newtheorem{Conjecture}[Theorem]{Conjecture}

\newtheorem{exercise}[Theorem]{Exercise}
\newtheorem{Question}[Theorem]{Question}
\newtheorem{lemma}[Theorem]{Lemma}
\newtheorem{property}[Theorem]{Property}
\newtheorem{proposition}[Theorem]{Proposition}
\newtheorem{ax}[Theorem]{Axiom}
\newtheorem{claim}[Theorem]{Claim}

\newtheorem{nTheorem}{Surjectivity Theorem}

\theoremstyle{definition}
\newtheorem{Definition}[Theorem]{Definition}
\newtheorem{problem}[Theorem]{Problem}
\newtheorem{question}[Theorem]{Question}
\newtheorem{Example}[Theorem]{Example}
\newtheorem{remark}[Theorem]{Remark}
\newtheorem{diagram}{Diagram}
\newtheorem{Remark}[Theorem]{Remark}
\newcommand{\diagref}[1]{diagram~\ref{#1}}
\newcommand{\thmref}[1]{Theorem~\ref{#1}}
\newcommand{\secref}[1]{Section~\ref{#1}}
\newcommand{\subsecref}[1]{Subsection~\ref{#1}}
\newcommand{\lemref}[1]{Lemma~\ref{#1}}
\newcommand{\corref}[1]{Corollary~\ref{#1}}
\newcommand{\exampref}[1]{Example~\ref{#1}}
\newcommand{\remarkref}[1]{Remark~\ref{#1}}
\newcommand{\corlref}[1]{Corollary~\ref{#1}}
\newcommand{\claimref}[1]{Claim~\ref{#1}}
\newcommand{\defnref}[1]{Definition~\ref{#1}}
\newcommand{\propref}[1]{Proposition~\ref{#1}}
\newcommand{\prref}[1]{Property~\ref{#1}}
\newcommand{\itemref}[1]{(\ref{#1})}
\newcommand{\ul}[1]{\underline{#1}}


\newcommand{\CE}{\mathcal{E}}
\newcommand{\CG}{\mathcal{G}}\newcommand{\CV}{\mathcal{V}}
\newcommand{\CL}{\mathcal{L}}
\newcommand{\CM}{\mathcal{M}}
\newcommand{\A}{\mathcal{A}}
\newcommand{\CO}{\mathcal{O}}
\newcommand{\B}{\mathcal{B}}
\newcommand{\CS}{\mathcal{S}}
\newcommand{\CX}{\mathcal{X}}
\newcommand{\CY}{\mathcal{Y}}
\newcommand{\CT}{\mathcal{T}}
\newcommand{\CW}{\mathcal{W}}
\newcommand{\CJ}{\mathcal{J}}

\newcommand{\st}{\sigma}
\renewcommand{\k}{\varkappa}
\newcommand{\Frac}{\mbox{Frac}}
\newcommand{\XC}{\mathcal{X}}
\newcommand{\wt}{\widetilde}
\newcommand{\wh}{\widehat}
\newcommand{\mk}{\medskip}
\renewcommand{\sectionmark}[1]{}
\renewcommand{\Im}{\operatorname{Im}}
\renewcommand{\Re}{\operatorname{Re}}
\newcommand{\la}{\langle}
\newcommand{\ra}{\rangle}
\newcommand{\LND}{\mbox{LND}}
\newcommand{\Pic}{\mbox{Pic}}
\newcommand{\lnd}{\mbox{lnd}}
\newcommand{\GLND}{\mbox{GLND}}\newcommand{\glnd}{\mbox{glnd}}
\newcommand{\Der}{\mbox{DER}}\newcommand{\DER}{\mbox{DER}}
\renewcommand{\th}{\theta}
\newcommand{\ve}{\varepsilon}
\newcommand{\1}{^{-1}}
\newcommand{\iy}{\infty}
\newcommand{\iintl}{\iint\limits}
\newcommand{\capl}{\operatornamewithlimits{\bigcap}\limits}
\newcommand{\cupl}{\operatornamewithlimits{\bigcup}\limits}
\newcommand{\suml}{\sum\limits}
\newcommand{\ord}{\operatorname{ord}}
\newcommand{\gal}{\operatorname{Gal}}
\newcommand{\bk}{\bigskip}
\newcommand{\fc}{\frac}
\newcommand{\g}{\gamma}
\newcommand{\be}{\beta}
\newcommand{\dl}{\delta}
\newcommand{\Dl}{\Delta}
\newcommand{\lm}{\lambda}
\newcommand{\Lm}{\Lambda}
\newcommand{\om}{\omega}
\newcommand{\ov}{\overline}
\newcommand{\vp}{\varphi}
\newcommand{\kap}{\varkappa}

\newcommand{\Vp}{\Phi}
\newcommand{\Varphi}{\Phi}
\newcommand{\BC}{\mathbb{C}}
\newcommand{\C}{\mathbb{C}}\newcommand{\BP}{\mathbb{P}}
\newcommand{\BQ}{\mathbb {Q}}
\newcommand{\BM}{\mathbb{M}}
\newcommand{\BR}{\mathbb{R}}\newcommand{\BN}{\mathbb{N}}
\newcommand{\BZ}{\mathbb{Z}}\newcommand{\BF}{\mathbb{F}}
\newcommand{\BA}{\mathbb {A}}
\renewcommand{\Im}{\operatorname{Im}}
\newcommand{\idd}{\operatorname{id}}
\newcommand{\ep}{\epsilon}
\newcommand{\tp}{\tilde\partial}
\newcommand{\doe}{\overset{\text{def}}{=}}
\newcommand{\supp} {\operatorname{supp}}
\newcommand{\loc} {\operatorname{loc}}
\newcommand{\de}{\partial}
\newcommand{\z}{\zeta}
\renewcommand{\a}{\alpha}
\newcommand{\G}{\Gamma}
\newcommand{\der}{\mbox{DER}}

\newcommand{\Spec}{\operatorname{Spec}}
\newcommand{\Sym}{\operatorname{Sym}}
\newcommand{\Aut}{\operatorname{Aut}}

\newcommand{\Idd}{\operatorname{Id}}

\newcommand{\tG}{\widetilde G}

\newcommand{\FX}{\mathfrac {X}}
\newcommand{\FV}{\mathfrac {V}}
\newcommand{\SX}{\mathcal {X}}
\newcommand{\SV}{\mathcal {V}}
\newcommand{\SO}{\mathcal {O}}
\newcommand{\SD}{\mathcal {D}}
\newcommand{\Sr}{\rho}
\newcommand{\SR}{\mathcal {R}}
\newcommand{\cl}{\mathcal{C}}
\newcommand{\ok}{\mathcal{O}_K}
\newcommand{\ab}{\mathcal{AB}}
\newcommand{\aut}{\text{Aut}}

\setcounter{equation}{0} \setcounter{section}{0}

\newcommand{\ds}{\displaystyle}
\newcommand{\gl}{\lambda}
\newcommand{\gL}{\Lambda}
\newcommand{\gge}{\epsilon}
\newcommand{\gG}{\Gamma}
\newcommand{\ga}{\alpha}
\newcommand{\gb}{\beta}
\newcommand{\gd}{\delta}
\newcommand{\gD}{\Delta}
\newcommand{\gs}{\sigma}
\newcommand{\mbq}{\mathbb{Q}}
\newcommand{\mbr}{\mathbb{R}}
\newcommand{\mbz}{\mathbb{Z}}
\newcommand{\mbc}{\mathbb{C}}
\newcommand{\mbn}{\mathbb{N}}
\newcommand{\mbp}{\mathbb{P}}
\newcommand{\mbf}{\mathbb{F}}
\newcommand{\mbe}{\mathbb{E}}
\newcommand{\lcm}{\text{lcm}\,}
\newcommand{\mf}[1]{\mathfrak{#1}}
\newcommand{\ol}[1]{\overline{#1}}
\newcommand{\mc}[1]{\mathcal{#1}}
\newcommand{\nequiv}{\equiv\hspace{-.07in}/\;}
\newcommand{\bnequiv}{\equiv\hspace{-.13in}/\;}

\title{A local-global principle for power maps}
\author{N. Jones}

\thanks{This work was partially supported by a Ralphe E. Powe Junior Faculty Enhancement award and also by NSA grant H98230-12-1-0210.  The author gratefully acknowledges Oak Ridge Associated Universities and the National Security Agency for this support.}

\date{}

\begin{abstract}
Let $f$ be a function from the set of rational numbers into itself.  We call $f$ a global power map if $f(\ga) = \ga^k$ for some integer exponent $k$.  We call $f$ a local power map at the prime number $p$ if $f$ induces a well-defined group homomorphism on the multiplicative group $(\mbz/p\mbz)^\times$.  We conjecture that if $f$ is a local power map at an infinite number of primes $p$, then $f$ must be a global power map.  Our main theorem implies that if $f$ is a local power map at every prime $p$ in a set with positive upper density relative to the set of all primes, then $f$ must be a global power map.  In particular, this represents progress towards a conjecture of Fabrykowski and Subbarao.
\end{abstract}

\maketitle

\section{Introduction} \label{introduction}

Let $\mbn := \{ 1, 2, 3, \dots \}$ denote the set of natural numbers and for any prime number $p$, define 
\[
\mbn_{(p)} := \{ n \in \mbn ; p \nmid n \}.  
\]
Given a function
\[
f : \mbn \longrightarrow \mbn,
\]
suppose that $p$ is a prime for which
\begin{equation*} 
f(\mbn_{(p)}) \subseteq \mbn_{(p)}.
\end{equation*}
For such a prime $p$, one may ask whether there exists a multiplicative group homomorphism $f_p : \mbf_p^\times \longrightarrow \mbf_p^\times$ for which the diagram
\begin{equation} \label{commutingdiagram}
\begin{CD}
\mbn_{(p)} @>f>> \mbn_{(p)} \\
@V\text{red}_pVV @V\text{red}_pVV \\
\mbf_p^\times @>f_p>> \mbf_p^\times
\end{CD}
\end{equation}
commutes.  We consider the set $S_f$ of such primes:
\begin{equation*} 
S_f := \{ p \text{ prime} ; \; f(\mbn_{(p)}) \subseteq \mbn_{(p)} \text{ and } \exists \text{ a homomorphism $f_p$ for which \eqref{commutingdiagram} commutes} \}.
\end{equation*}
For instance, if $f(n) = n^2$ for all $n \in \mbn$, then one has $S_f = \{ \text{all primes} \}$.  On the other hand, suppose $f$ is defined by
\begin{equation*} 
f(n) = 
\begin{cases}
7 & \text{ if } n \equiv 0 \pmod{3} \\
3 \pi(n) + 1 & \text{ if } n \equiv 1 \pmod{3} \\
3 \nu(n) + 2 & \text{ if } n \equiv 2 \pmod{3},
\end{cases}
\end{equation*}
where (here and throughout the paper) $\pi(n) := | \{ \text{primes $p$ } : \; p \leq n \} |$ and $\nu(n) := | \{ \text{primes $p$} : \; p \text{ divides } n \}|$.  Then $3 \in S_f$, and quite probably $S_f = \{ 3 \}$.  By using a diagonalization argument, one can construct functions $f$ which certainly satisfy $S_f = \{ 3\}$.  By incorporating the Chinese Remainder Theorem, given any \emph{finite} set of primes $S$ one can find a function $f$ for which $S_f = S$.  
Our motivating question is the following.
\begin{Question} \label{questionone}
Does there exist a function $f : \mbn \longrightarrow \mbn$ for which $S_f$ is infinite and $S_f \neq \{ \text{all primes} \}$?
\end{Question}

We will presently couch this question in slightly different terms.  
Returning to \eqref{commutingdiagram}, note that since $\mbf_p^\times$ is cyclic, any multiplicative homomorphism $f_p : \mbf_p^\times \longrightarrow \mbf_p^\times$ must be of the form 
$
f_p(x) = x^{k_p} 
$
for some exponent $k_p \in \mbz/(p-1)\mbz$.  In particular,
\begin{equation*} 
S_f = \{ p \text{ prime} ; \; \exists k_p \in \mbz/(p-1)\mbz \; \text{ for which } \forall n \in \mbn_{(p)}, \quad f(n) \equiv n^{k_p} \pmod{p} \}.
\end{equation*}
\begin{Definition}
Let $S$ be a set of prime numbers.  A function $f : \mbn \longrightarrow \mbn$ is a \textbf{local power map at $S$} if $S \subseteq S_f$.
\end{Definition}
\begin{Definition}
A function $f : \mbn \longrightarrow \mbn$ is called a \textbf{global power map} if there is an exponent $k \in \mbn \cup \{ 0 \}$ such that, for each $n \in \mbn$ one has $f(n) = n^k$.
\end{Definition}
Note that any global power map $f$ is a local power map at $S = \{ \text{all primes} \}$, and so in particular is a local power map at an infinite set of primes.  The main result of \cite{heathbrown} (or even of its predecessor \cite{guptamurty}), implies that there exists a prime number $q$ for which the set
\[
\{ p \text{ prime}, \; \langle q \pmod{p} \rangle = \mbf_p^\times \}
\]
is infinite.  Using this fact, one may deduce that any local power map at $S = \{ \text{all primes}\}$ must be a global power map.
Thus, Question \ref{questionone} may be stated equivalently as
\begin{Question}
Does there exist a function $f : \mbn \longrightarrow \mbn$ which is a local power map at an \emph{infinite} set of primes, but which is not a global power map?
\end{Question}
We will provide heuristics which lead us to conjecture that the answer is no:
\begin{Conjecture} \label{mainconjecture}
Suppose $f : \mbn \longrightarrow \mbn$ is a local power map at an infinite set of primes (i.e. suppose that $| S_f | = \infty$).   Then $f$ must be a global power map.
\end{Conjecture}
The preceding discussion applies just as well when we replace $\mbn$ with the set $\mbq$ of rational numbers.  Indeed, for any prime number $p$ let us employ the standard notation
\[
\begin{split}
\mbz_{(p)}  := & \left\{ \frac{n}{m} \in \mbq : \;  p \nmid m \right\}, \\
\mbz_{(p)}^\times = & \left\{ \frac{n}{m} \in \mbz_{(p)} : \; p \nmid n \right \},
\end{split}  
\]
where the fraction $n/m$ above is assumed to be in lowest terms.  As is well-known, $\mbz_{(p)}$ is a local ring with maximal ideal $p \mbz_{(p)}$, and one has an isomorphism
\[
\frac{\mbz_{(p)}}{p \mbz_{(p)}}
\; \simeq \;
\frac{\mbz}{p \mbz}.
\]
For $\ga, \gb \in \mbz_{(p)}$, we write $\ga \equiv \gb \pmod{p}$ provided $\ga - \gb \in p \mbz_{(p)}$.  Now if
\[
f : \mbq \longrightarrow \mbq
\]
is any function, we consider the set
\begin{equation*} 
S_f := \{ p \text{ prime} ; \; \exists k_p \in \mbz/(p-1)\mbz \; \text{ such that } \forall \ga \in \mbz_{(p)}^\times, \quad f(\ga) \equiv \ga^{k_p} \pmod{p} \}.
\end{equation*}
\begin{Definition}
Let $S$ be a set of prime numbers.  A function $f : \mbq \longrightarrow \mbq$ is a \textbf{local power map at $S$} if $S \subseteq S_f$, and $f$ is a \textbf{global power map} if, for some integer $k \in \mbz$, one has
\[
\forall \ga \in \mbq, \; f(\ga) = \ga^k.
\]
\end{Definition} 
We make the following conjecture, which (as we will show in Section \ref{previousresultssection}) implies Conjecture \ref{mainconjecture}.
\begin{Conjecture} \label{generalmainconjecture}
Suppose that $f : \mbq \longrightarrow \mbq$ is a local power map at an infinite set of primes (i.e. suppose that $| S_f | = \infty$).   Then $f$ must be a global power map.
\end{Conjecture}

\begin{remark}
Conjecture \ref{generalmainconjecture} also implies Conjecture 3.1 of \cite{subbaraofabrykowski}, which states that any quasi-multiplicative function $f :  \mbn \longrightarrow \mbz$ which is not identically zero and satisfies
\[
f(n+p) \equiv f(n) \pmod{p}
\]
for infinitely many primes $p$ must be a global power map.  This connection will be discussed in more detail in Section \ref{previousresultssection}.
\end{remark}

In the present paper, we will prove the following weakened version of Conjecture \ref{generalmainconjecture}, in which ``$S_f$ is infinite'' is replaced by ``$S_f$ has positive upper density in the primes.''  For any set $S$ of prime numbers, define
\[
S(x) := \{ p \in S : \; p \leq x \}
\]
and the upper density
\[
\ol{\gd}(S) := \limsup_{x \rightarrow \infty} \frac{ | S(x) |}{ \pi(x) }.
\]
We will prove the following theorem.
\begin{Theorem} \label{maintheorem}
Let $f : \mbq \longrightarrow \mbq$ be any function which is not a global power map.  Then there exist real constants $b_f, c_f > 0$ so that for $x \geq c_f$, the bound
\[
| S_f(x) | \; \ll \; \frac{\log \log \log \log x}{\log \log \log x} \cdot \pi(x) + b_f
\]
holds, with an absolute implied constant.  In particular, if $f : \mbq \longrightarrow \mbq$ is a function for which $\ol{\gd}(S_f) > 0$, then $f$ is a global power map.
\end{Theorem}

Our proof of this theorem applies an effective version of the Chebotarev density theorem of Lagarias and Odlyzko to certain Kummer extensions attached to the function $f$.

\section{Notation}

Throughout the paper, we will use the following notation.  For $\ga \in \mbq$ and a prime number $p$, there is a unique integer $c$ for which $\ga = p^c \cdot (a/b)$, where $a, b \in \mbz$ and $p \nmid ab$.  We then define $\ord_p(\ga) := c$.  Furthermore, we define $\num(\ga) := n$ and $\den(\ga) := m$ where $\ga = n/m \in \mbq$ is written in lowest terms.  We use the symbols $O(\cdot)$ and $\ll$ in the usual ways, namely if $f, g \colon [\gamma, \infty) \rightarrow \mbc$ are complex functions then we write
\[
f = O(g), \quad \text{ or equivalently } \quad f \ll g
\]
if there is a positive constant $C$ for which $| f(x) | \leq C | g(x) |$ for all $x \in [\gamma, \infty)$.  In case there is an auxiliary parameter $y$ upon which the implied constant $C$ depends, we will indicate this with a subscript, so that
\[
f = O_y(g) \quad \text{ or equivalently } \quad f \ll_y g
\]
is used to indicate that $| f(x) | \leq C(y) | g(x) |$, where the $C(y)$ may depend on $y$ but not on $x$.  We write $f(x) \sim g(x)$ as $x \rightarrow \infty$ to mean that $f(x)$ is asymptotic to $g(x)$ as $x \rightarrow \infty$, i.e. to mean that $\ds \lim_{x\rightarrow \infty} f(x)/g(x) = 1$.
When used as variables, the letters $p$ and $\ell$ will always denote prime numbers.  We will occasionally denote the reduction modulo $p$ map by
\[
\begin{split}
\mbz_{(p)} &\rightarrow \mbf_p \\
n &\mapsto \hat{n}.
\end{split}
\]
For an odd prime number $\ell$, let $\zeta_\ell$ denote a primitive $\ell$-th root of unity.  In our discussion of Kummer extensions, we will employ the following vector notation.  For $m \geq 0$ and ${\bf{c}} = (c_1, c_2, \dots, c_m) \in (\mbq^\times)^m$, we define 
\[
\mbq(\zeta_\ell, {\bf{c}}^{1/\ell}) := \mbq(\zeta_\ell, c_1^{1/\ell}, c_2^{1/\ell}, \dots, c_m^{1/\ell}),
\]
where if $m = 0$ we make the interpretation $\mbq(\zeta_\ell, {\bf{c}}^{1/\ell}) := \mbq(\zeta_\ell)$.  Furthermore, for a vector ${\bf{n}} = (n_1, n_2, \dots, n_m) \in \mbz^m$, we will use the notation
\[
{\bf{c}}^{{\bf{n}}} := \prod_{i = 1}^m c_i^{n_i} \in \mbq^\times.
\]

\section{Related results} \label{previousresultssection}

We now give a brief survey of various related results (each with slightly different hypotheses on the integer-valued function $f$,
but with the conclusion ``then $f$ is a global power map,'' or a closely related conclusion).  Before doing so, let us make a few elementary observations and show why Conjecture \ref{generalmainconjecture} implies Conjecture \ref{mainconjecture}.

\begin{lemma} \label{completelymultiplicativelemma}
Suppose that $f : \mbq \longrightarrow \mbq$ is a function for which $S_f$ is infinite.  Then 
\begin{equation} \label{fpreservesQtimes}
f(\mbq^\times) \subseteq \mbq^\times, 
\end{equation}
and the restriction of $f$ to $\mbq^\times$ is completely multiplicative, i.e. for any $\ga, \gb \in \mbq^\times$ one has 
\begin{equation} \label{fmultiplicative}
f(\ga \gb) = f(\ga)f(\gb).
\end{equation}
\end{lemma}
\begin{proof}
To prove \eqref{fpreservesQtimes}, fix $\ga \in \mbq^\times$.  If $f(\ga) = 0$ then for each prime $p$,
\begin{equation} \label{fneq0}
\ord_p(\ga) = 0
\; \Longrightarrow \;
p \notin S_f, 
\end{equation}
implying that $S_f$ is finite, a contradiction.  Thus, \eqref{fpreservesQtimes} holds.  The second assertion \eqref{fmultiplicative} follows from the observation that, for any $\gamma \in \mbq$,
\begin{equation} \label{equiv0modinfinitelymanyp}
\gamma \equiv 0 \pmod{p} \; \text{ for infinitely many primes $p$}
\; \Longrightarrow \;
\gamma = 0,
\end{equation}
which is true since if $\gamma = a/b$ in lowest terms then $\gamma \equiv 0 \pmod{p}$ if and only if $p$ divides $a$.  

To prove that $f$ is completely multiplicative, fix $\ga, \gb \in \mbq^\times$ and apply \eqref{equiv0modinfinitelymanyp} to $\gamma = f(\ga\gb) - f(\ga)f(\gb)$, which is divisible by every prime $p \in S_f$ for which $\ord_p(\ga) = \ord_p(\gb) = 0$.  Since $S_f$ is infinite, there are infinitely many such primes $p$.
This concludes the proof.
\end{proof}

In particular, if $f : \mbq \longrightarrow \mbq$ and $S_f$ is infinite, then $f|_{\mbq^\times}$ is uniquely determined by its values on $\{ -1 \} \cup \{ \text{all primes} \}$.
We will now show why Conjecture \ref{generalmainconjecture} implies Conjecture \ref{mainconjecture}, which amounts to proving the following lemma.  Since we will be varying a bit the domain of the function $f$, let us first write down the general situation, which encapsulates the set-up in both of the conjectures given in the introduction.  

If $A \subseteq \mbq$ is a subset which is closed under multiplication, then the set
\[
A_{(p)} := A \cap \mbz_{(p)}^\times
\]
is also closed under multiplication.  Furthermore, if
\[
f : A \longrightarrow \mbq
\]
is any function, then we may define the set $S_f$ of primes as before by
\[
S_f := \{ p \text{ prime} ; \; \exists k_p \in \mbz/(p-1)\mbz \; \text{ such that } \forall \ga \in A_{(p)}, \quad f(\ga) \equiv \ga^{k_p} \pmod{p} \}.
\]
\begin{lemma} \label{extendingflemma}
Suppose that $f : \mbn \longrightarrow \mbq$ is any function for which $S_f$ is infinite.  Then there is a completely multiplicative function 
\[
\tilde{f} : \mbq \longrightarrow \mbq
\] 
such that $\forall n \in \mbn$, $\tilde{f}(n) = f(n)$ and for which $S_{\tilde{f}}$ is infinite.
\end{lemma}

\begin{proof}
First of all, by the same reasoning as in \eqref{fneq0}, the infinitude of $S_f$ implies that
\begin{equation} \label{fofnnotzero}
\forall n \in \mbn, \; f(n) \neq 0.
\end{equation}
Furthermore, by the same reasoning as in the proof of \eqref{fmultiplicative} one sees that $f$ is completely multiplicative.  In particular,
\begin{equation*} 
f(1) = 1.
\end{equation*}

We begin by extending $f$ to a function $f_1 : \mbz \longrightarrow \mbq$.   Note that for odd $p \in S_f$, since $k_p \in \mbz/(p-1)\mbz$, the parity of $k_p$ is well-defined, and by the pigeon-hole principle, either $k_p$ is infinitely often even or it is infinitely often odd.  We set
\[
\nu_f := 
\begin{cases}
0 & \text{ if $k_p$ is even for infinitely many $p \in S_f$} \\
1 & \text{ otherwise,}
\end{cases}
\]
and then define $f_1 : \{ -1, 0, 1 \} \longrightarrow \{ -1, 0, 1 \}$ by
\[
f_1( \pm 1) := (\pm 1)^{\nu_f} \quad \text{ and } \quad f_1(0) := 0.
\]
Then, for each $x,y \in \{ -1, 0, 1 \}$ one has $f_1(xy) = f_1(x) f_1(y)$.  Furthermore, if
 $\sgn : \mbz \longrightarrow \{ -1, 0, 1 \}$ is defined by
\[
\sgn(n) := 
\begin{cases}
\frac{n}{|n|} & \text{ if } n \neq 0 \\
0 & \text{ if } n = 0,
\end{cases}
\]
then any $n \in \mbz$ decomposes as $n = \sgn(n) \cdot |n|$, and we define 
\[
f_1(n) := f_1(\sgn(n)) \cdot f( |n| ).
\]
It follows that $f_1 : \mbz \longrightarrow \mbq$ is completely multiplicative and (by \eqref{fofnnotzero}) satisfies
\begin{equation} \label{fofnstillnotzero}
n \neq 0
\; \Longrightarrow \;
f_1(n) \neq 0.
\end{equation}
Furthermore,
\[
S_{f_1} \supseteq \{ p \in S_f : (-1)^{k_p} = (-1)^{\nu_f} \},
\]
and by construction the right-hand set is infinite.  We now extend $f_1$ to all of $\mbq$ by setting
\[
\tilde{f}\left( \frac{n}{m} \right) := \frac{f_1(n)}{f_1(m)}.
\]
Since $f_1$ is completely multiplicative (and by \eqref{fofnstillnotzero}), $\tilde{f}$ is well-defined, is completely multiplicative, and satisfies $S_{f_1} \subseteq S_{\tilde{f}}$.  This proves the lemma.
\end{proof}

By the Lemma \ref{completelymultiplicativelemma}, one may as well add ``$f$ is completely multiplicative'' to the hypothesis of Conjecture \ref{generalmainconjecture}.  More generally, recall that $f$ is called \emph{multiplicative} if $f(nm) = f(n)f(m)$ whenever $\gcd(m,n) = 1$.

It follows from a result of P. Erd\H{o}s \cite[Theorem V]{erdos} that
\[
\begin{pmatrix}
\text{$f : \mbn \longrightarrow \mbn$ is multiplicative} \\
\text{and } \forall n \in \mbn, \, f(n+1) \geq f(n)
\end{pmatrix}
\; \Longrightarrow \;
f \text{ is a global power map.}
\]
Replacing the monotonicity hypothesis with the condition
\begin{equation} \label{fwelldefinedmodk}
\forall n \in \mbn, \quad f(n + k) \equiv f(n) \pmod{k},
\end{equation}
M. V. Subbarao \cite{subbarao} has shown that
\[
\begin{pmatrix}
\text{$f : \mbn \longrightarrow \mbz$ is multiplicative} \\
\text{and satisfies \eqref{fwelldefinedmodk} for all $k \in \mbn$ }
\end{pmatrix}
\; \Longrightarrow \;
\begin{pmatrix}
\text{$f$ is a global power map} \\
\text{ or $f(n) = 0 \quad \forall n \in \mbn$.}
\end{pmatrix}
\]

In \cite{subbaraofabrykowski} Subbarao and J. Fabrykowski prove a similar theorem, with the multiplicativity of $f$ relaxed a bit (as we presently describe), and where \eqref{fwelldefinedmodk} is only demanded for \emph{primes} $k$, i.e.
\begin{equation} \label{fwelldefinedmodp}
\forall n \in \mbn, \quad f(n + p) \equiv f(n) \pmod{p}.
\end{equation}
The following is equivalent to \cite[Definition 1.3]{subbaraofabrykowski} 
\begin{Definition} \label{quasimultiplicativedef}
 A function $f : \mbn \longrightarrow \mbn$ is called \emph{quasi-multiplicative} if, for any $n \in \mbn$ and any prime $p$ not dividing $n$, one has
\[
f(pn) = f(p)f(n).
\]
\end{Definition}
For any function $f : \mbn \longrightarrow \mbz$, let us define the set
\begin{equation} \label{defofTf}
T_f := \{ p \text{ prime} ; \; \text{ \eqref{fwelldefinedmodp} holds} \}.
\end{equation}
In \cite{subbaraofabrykowski} it is shown that
\[
\begin{pmatrix}
\text{$f : \mbn \longrightarrow \mbz$ is quasi-multiplicative} \\
\text{and $T_f = \{ \text{all primes} \}$ }
\end{pmatrix}
\; \Longrightarrow \;
\begin{pmatrix}
\text{$f$ is a global power map} \\
\text{ or $f(n) = 0 \quad \forall n \in \mbn$.}
\end{pmatrix}
\]
Furthermore, they make the following conjecture. 
\begin{Conjecture} \label{theirconjecture}
If $f : \mbn \longrightarrow \mbz$ is quasi-multiplicative and $T_f$ is infinite, then either  $f(n) = 0$ for each $n \in \mbn$ or $f$ is a global power map.
\end{Conjecture}
The next lemma, taken together with Lemma \ref{extendingflemma}, shows that Conjecture \ref{theirconjecture} is implied by Conjecture \ref{generalmainconjecture}.  Note that, for any $p \in T_f$, there is a well-defined function
\[
f_p : \mbf_p \longrightarrow \mbf_p, \quad\quad f_p(n \pmod{p}) := f(n).
\]
\begin{lemma} \label{conjectureimpliesconjecturelemma}
Suppose that $f : \mbn \longrightarrow \mbz$ is quasi-multiplicative and that $T_f$ is infinite.  Then either $f(n) = 0$ for each $n \in \mbn$, or $T_f - (S_f \cap T_f)$ is finite (and thus $S_f$ is infinite).
\end{lemma}
\begin{proof}
Fix any prime $p \in T_f$ and note that $p \in S_f \cap T_f$ if and only if
\begin{equation} \label{fppreservesmultgroup}
f_p(\mbf_p^\times) \subseteq \mbf_p^\times
\end{equation}
holds and $f_p$ is a multiplicative homomorphism.  Let
\[
\begin{split}
\mbz & \rightarrow \mbf_p \\
n & \mapsto \hat{n}
\end{split}
\]
denote the reduction modulo $p$ map and choose $g \in \mbn$ so that
$
\langle \hat{g} \rangle = \mbf_p^\times.
$
Suppose that \eqref{fppreservesmultgroup} does not hold, i.e. that $f_p(\hat{g}^n) = \hat{0}$ for some positive integer $n$.  By Dirichlet's theorem on primes in arithmetic progressions, one may find $n$ prime numbers $q_1, q_2, \dots, q_n$ for which
\[
\forall i \in \{1, 2, \dots, n \}, \quad q_i \equiv g \pmod{p}.
\]
It follows from Definition \ref{quasimultiplicativedef} that 
\begin{equation} \label{dirichlettrick}
\hat{0} = f_p \left( \hat{g}^n \right) = f_p \left( \prod_{i=1}^n \hat{q}_i \right) = \prod_{i=1}^n f_p( \hat{q}_i ) = \left( f_p(\hat{g}) \right)^n,
\end{equation}
and so we conclude that, for any prime $p \in T_f$,
\[
\text{condition \eqref{fppreservesmultgroup} fails }
\; \Longrightarrow \;
f_p (\mbf_p) = \{ \hat{0} \}.  
\]
Furthermore, if we set
\[
T_0 := \{ p \in T_f : \; f_p( \mbf_p ) = \{ \hat{0} \} \},
\]
then for each $n \in \mbn$, $f(n)$ is divisible by every prime $p \in T_0$.  Thus, 
\[
|T_0| = \infty
\; \Longrightarrow \;
\forall n \in \mbn, \; 
f(n) = 0.
\]
Assuming $f$ is not identically zero, we have that $T_0$ is finite, and putting $S := T_f - T_0$, we see that \eqref{fppreservesmultgroup} holds for each $p \in S$.  Furthermore, using Dirichlet's theorem on primes in arithmetic progressions and reasoning as in \eqref{dirichlettrick}, one sees that the restriction of $f_p$ to $\mbf_p^\times$ is a multiplicative homomorphism for each $p \in S$.  In particular, $S = T_f \cap S_f$, which concludes the proof.
\end{proof}
\begin{remark}
Lemmas \ref{conjectureimpliesconjecturelemma} and \ref{completelymultiplicativelemma} together imply that any quasi-multiplicative function satisfying \eqref{fwelldefinedmodp} for infinitely many primes $p$ is necessarily completely multiplicative, solving Problem 3.9 of \cite{subbaraofabrykowski}.
\end{remark}

The main result of \cite{feherphong} implies that, if the set $\{ \text{all primes} \} - T_f$ is finite, then either $f$ is identically zero or $f$ is a global power map.  A somewhat stronger result may be found in \cite[Proposition 1, p. 329]{khareprasad} (whose proof appeals to \cite[Theorem 1]{corralesschoof}), which implies that if $T_f$ has density one in the set of primes, then either $f$ is identically zero or $f$ is a global power map.
Putting Lemmas \ref{conjectureimpliesconjecturelemma} and \ref{extendingflemma} together with Theorem \ref{maintheorem}, we obtain the following corollary, which represents further progress towards Conjecture \ref{theirconjecture}.  
\begin{cor}
Let $f : \mbn \longrightarrow \mbz$ be a quasi-multiplicative function and let $T_f$ be defined by \eqref{defofTf}.  Then either $f$ is identically zero, or $f$ is a global power map, or there exist real constants $b_f, c_f > 0$ so that, for $x \geq c_f$, the bound
\[
| T_f(x) | \; \ll \; \frac{\log \log \log \log x}{\log \log \log x} \cdot \pi(x) + b_f
\]
holds, with an absolute implied constant.  In particular, if $f : \mbn \longrightarrow \mbz$ is a quasi-multiplicative function for which $\ol{\gd}(T_f) > 0$, then either $f$ is identically zero or $f$ is a global power map.
\end{cor}

Returning to our survey of related results, one may also replace the assumption of (quasi-)multiplicativity of $f$ by upper bounds on its growth.  In this spirit, I. Ruzsa \cite{ruzsa} proved that, if $f : \mbn \longrightarrow \mbz$ satisfies \eqref{fwelldefinedmodk} for each $k \in \mbn$ together with the bound
\[
|f(n)| \ll (e - 1)^{\ga n}
\]
for some $\ga < 1$, then $f$ is a polynomial map.  Ruzsa also conjectured that the same result should hold with $e-1$ replaced by $e$, and some progress on this conjecture has been made by Zannier \cite{zannier}.

Viewed more broadly, Conjecture \ref{generalmainconjecture} asserts that, if a function $f$ has some special form when reduced modulo $p$ for infinitely many primes $p$,  then $f$ itself must have a special form.  We remark that, in other contexts, one may find results of this type; see for instance \cite[Theorem 4, pp. 329--330]{khareprasad}.

\section{Heuristics}

We will now provide a probabilistic argument to support Conjecture \ref{generalmainconjecture}.  We begin with some preliminary observations.

\begin{lemma} \label{noteventuallyconstantpowerlemma}
Suppose that $f : \mbq \longrightarrow \mbq$ is a function for which $|S_f| = \infty$.  If there is an exponent $k \in \mbz$ and a constant $C$ for which $f(q) \in \{ q^k, - q^k \}$ for all primes $q \geq C$, then $f(\ga) = \ga^k$ for all $\ga \in \mbq$.
\end{lemma}
\begin{proof}
By Lemma \ref{completelymultiplicativelemma}, $f(\mbq^\times) \subseteq \mbq^\times$ and $f$ is completely multiplicative, upon which it follows that $f(-1) \in \{ 1, -1 \}$.  Thus,
\begin{equation} \label{decompoff}
\begin{split}
\forall \ga \in \mbq^\times, \quad f(\ga) 
&= f(\sgn(\ga)) \cdot f(| \ga |) \\
&= \pm f(| \ga |),
\end{split}
\end{equation}
where $\sgn (\ga) := \ga / | \ga |$ denotes the sign of $\ga$.  One sees that $f(\ga)$ is determined by $f(\sgn(\ga))$ and its restriction
\[
f : \mbq_+^\times \longrightarrow \mbq^\times
\]
to $\mbq_+^\times$.
Assuming that $f(q) = \pm q^k$ for some $k \in \mbz$ and all primes $q \geq C$, then define
\[
\mu : \mbq_+^\times \longrightarrow \mbq^\times, \quad \mu(\ga) := \frac{f(\ga)}{\ga^k}.
\]
One checks that 
\begin{equation} \label{SfinSmu}
S_f \subseteq S_\mu, 
\end{equation}
and also that 
\[
\mu ( \{ q \text{ prime}; \; q \geq C \}) \subseteq \{ \pm 1 \}.  
\]
Either there exists a constant $C_1$ for which $\mu ( \{ q \text{ prime}; \; q \geq C_1 \}) = \{ 1 \}$, or for each constant $C_1$ one has $\mu ( \{ q \text{ prime}; \; q \geq C_1 \}) = \{ 1, -1 \}$ (since $f_p$ is a homomorphism, one cannot have $\mu ( \{ q \text{ prime}; \; q \geq C_1 \}) = \{ -1 \}$).
In the first case, by taking a large prime $q$ which is a primitive root modulo $p$, one finds that $k_p = k$ for each $p \in S_f$.  Thus for any $\ga \in \mbq_+^\times$, $f(\ga) - \ga^k$ is divisible by infinitely many primes $p \in S_f$, and so $f(\ga) = \ga^k$.

If on the other hand $\mu ( \{ q \text{ prime}; \; q \geq C_1 \}) = \{ 1, -1 \}$ for any constant $C_1$, then for each $p \in S_\mu$ and each prime $q$ which is large enough,
\[
\mu(q) = \left( \frac{q}{p} \right),
\]
the Legendre symbol at $p$.  If there are distinct primes $p_1, p_2 \in S_\mu$, then by Dirichlet's theorem on primes in arithmetic progressions, one may find a prime $q$ with
\[
\left( \frac{q}{p_1} \right) \neq \left( \frac{q}{p_2} \right),
\]
a contradiction.  Thus, in this case $|S_\mu| \leq 1$, contradicting \eqref{SfinSmu}.  Thus, we see that
\[
\forall \ga \in \mbq_+^\times, \quad f(\ga) = \ga^k.
\]
It follows from \eqref{decompoff} that
\[
\forall \ga \in \mbq^\times, \quad f(\ga) = 
\begin{cases}
	\ga^k &\text{ or } \\
	\sgn(\ga) \ga^k, & 
\end{cases}
\]
according to whether or not $f(-1) = (-1)^k$.  If $f(\ga) = \sgn(\ga) \ga^k$, then
\[
\sgn(\ga) = \frac{f(\ga)}{\ga^k},
\]
and as before we conclude that $S_f \subseteq S_{\sgn}$, which contradicts the fact that $S_{\sgn} = \{ 2 \}$.  Therefore $f(\ga) = \ga^k$ for every $\ga \in \mbq^\times$, finishing the proof of the lemma.
\end{proof}

\begin{cor} \label{threenumberslemma}
Suppose that $f : \mbq \longrightarrow \mbq$ satisfies $| S_f | = \infty$.  Then either $f$ is a global power map, or for each $L \in \mbn$, one may find a set
\[
\mathfrak{N}_L = \{ n_1, n_2, \dots, n_L \} \subseteq \mbn
\]
of $L$ positive square-free integers satisfying
\[
\forall n \in \mathfrak{N}_L \quad f(n) \notin n^\mbz \cup - n^\mbz.
\]
\end{cor}
\begin{proof}
Suppose that $| S_f | = \infty$ but $f$ is not a global power map.  We proceed by induction on $L$.   For the base case $L=1$, either there exists a prime $q$ for which $f(q) \notin q^\mbz \cup - q^\mbz$ (in which case we set $n = q$), or else for each prime $q$, $f(q) \in \{ q^{k_q}, -q^{k_q} \}$ for some exponent $k_q \in \mbz$.  In the latter case, provided $f$ is not a global power map, then by Lemma \ref{noteventuallyconstantpowerlemma} one may find two primes $q_1$ and $q_2$ for which $k_{q_1} \neq k_{q_2}$.  By Lemma \ref{completelymultiplicativelemma}, $f$ must be completely multiplicative, and thus $f(q_1 q_2) \notin (q_1 q_2)^\mbz \cup  - (q_1 q_2)^\mbz$, so in this case we may set $n = q_1q_2$.   For the induction step, we reason the same way:  having constructed $\mathfrak{N}_{L-1}$, either there exists a prime $q$ larger than any $n \in \mathfrak{N}_{L-1}$ for which $f(q) \notin q^\mbz \cup - q^\mbz$ (in which case we set $\mathfrak{N}_L := \mathfrak{N}_{L-1} \cup \{ q \}$) or else for each prime $q$ larger than any $n \in \mathfrak{N}_{L-1}$, $f(q)  \in \{ q^{k_q}, - q^{k_q} \}$.  In the second case, by Lemma \ref{noteventuallyconstantpowerlemma} we may find two primes $q_1$ and $q_2$, each larger than any $n \in \mathfrak{N}_{L-1}$ and for which $k_{q_1} \neq k_{q_2}$, and we put $\mathfrak{N}_L := \mathfrak{N}_{L-1} \cup \{ q_1 q_2 \}$.
\end{proof}

Now suppose that $f : \mbq \longrightarrow \mbq$ is not a global power map, but nevertheless $S_f$ is infinite.  We presently apply probabilistic reasoning to deduce a (heuristic) contradiction.  Applying Corollary \ref{threenumberslemma} with $L = 3$, we may find three natural numbers $n_1, n_2, n_3 \in \mbn$ such that for each $i \in \{ 1, 2, 3 \}$,  $f(n_i) \notin n_i^\mbz \cup - n_i^\mbz$.  Consider the rational vectors
\[
{\bf{n}} := (n_1, n_2, n_3) \in (\mbq^\times)^3 \quad \text{and} \quad f({\bf{n}}) := (f(n_1),f(n_2),f(n_3)) \in (\mbq^\times)^3
\]
and define the sets $\Omega_{{\bf{n}}} \subseteq (\mbq^\times)^3$, $\Omega_{{\bf{n}}}(p) \subseteq \mbf_p^3$ by
\[
\begin{split}
\Omega_{\bf{n}} :=& \{ ( \ve_1 n_1^k, \ve_2 n_2^k, \ve_3 n_3^k) : \; k \in \mbz, \; \ve_i \in \{ \pm 1 \} \} \subseteq (\mbq^\times)^3 \\
\Omega_{{\bf{n}}}(p) :=& \{ (\hat{n}_1^k, \hat{n}_2^k, \hat{n}_3^k) : \; k \in \mbz/(p-1)\mbz \} \subseteq \mbf_p^3.
\end{split}
\]
By construction, we have that
\begin{equation} \label{globallynotin}
f({\bf{n}}) \notin \Omega_{\bf{n}}.
\end{equation}
For an arbitrary prime $p$ for which ${\bf{n}}, f({\bf{n}}) \in (\mbz_{(p)}^\times)^3$, we consider the reduction $f({\bf{n}}) \pmod{p} \in (\mbf_p^\times)^3$.  We evidently have
\begin{equation} \label{pinomega}
p \in S_f \; \Longrightarrow \; f({\bf{n}}) \pmod{p} \in \Omega_{\bf{n}}(p) \cap (\mbf_p^\times)^3.
\end{equation}
Thus for any prime p, we are motivated to ask how likely it is that
\begin{equation} \label{locallyin}
f({\bf{n}}) \pmod{p} \in \Omega_{\bf{n}}(p) \cap (\mbf_p^\times)^3.
\end{equation}
By virtue of \eqref{globallynotin}, it is reasonable to expect $f({\bf{n}})$ to behave like a random vector\footnote{Note that, if $f({\bf{n}}) \in \Omega_{{\bf{n}}}$, then $f({\bf{n}}) \pmod{p} \in \Omega_{{\bf{n}}}(p)$ for infinitely many primes $p$.  Indeed, for any $p$ satisfying $\left( \frac{n_i}{p} \right) = \ve_i$ for $i \in \{ 1, 2, 3\}$, one has $f({\bf{n}}) \pmod{p} \in \Omega_{{\bf{n}}}(p)$.} in $(\mbf_p^\times)^3$, at least with respect to lying in $\Omega_{\bf{n}}(p) \cap (\mbf_p^\times)^3$.  The heuristic probability that \eqref{locallyin} occurs is thus
\[
\text{Prob}\left( f({\bf{n}}) \pmod{p} \in \Omega_{\bf{n}}(p) \cap (\mbf_p^\times)^3 \right) \approx \frac{| \Omega_{\bf{n}}(p) \cap (\mbf_p^\times)^3 |}{| (\mbf_p^\times)^3 |} \leq \frac{(p-1)}{(p-1)^3}.
\]
Thus, by \eqref{pinomega}, the ``event'' $p \in S_f$ should occur with probability no greater than $\ds \frac{1}{(p-1)^2}$, and so it is expected that
\[
| \{ p \in S_f ; \; p \leq X \} | \leq \sum_{p \leq X} \frac{1}{(p-1)^2}.
\]
Since the right hand side is uniformly bounded in $X$, we expect $S_f$ to be finite, contradicting our assumption that $S_f$ is infinite.  This leads us to Conjecture \ref{generalmainconjecture}.

\section{Proof of main theorem}

The rest of the paper is devoted to a proof of Theorem \ref{maintheorem}.  We begin by observing that, for any parameters $0 \leq Y < Z$, one may bound the quantity $|S_f(x)|$ by two sums:
\begin{equation} \label{initialdecomposition}
| S_f(x) | \leq \sum_{{\begin{substack} {p \leq x \\ \forall \ell \in [Y, Z), \\  p \nequiv 1 \pmod{\ell}} \end{substack}}} 1 \; + \; \sum_{Y \leq \ell < Z} \sum_{{\begin{substack} {p \in S_f(x) \\ p \equiv 1 \pmod{\ell}} \end{substack}}} 1.
\end{equation}
We will eventually choose $Y = Y(x)$ and $Z = Z(x)$ appropriately so as to bound each of these sums.  

The main ingredient in our proof is an effective version of the Chebotarev density theorem, which will be discussed in general in Section \ref{effectivechebotarevdensitysection}.  It will be applied in the context of cyclotomic extensions to handle the first sum, and in the context of Kummer extensions to handle the second sum.  
The former ``cyclotomic part'' forms the content of Sections \ref{cyclotomicextensionssection} and \ref{proofoffirstpropositionsection}, while the latter ``Kummer extension'' part comprises Sections \ref{kummerextensionssection} and \ref{proofofsecondpropositionsection}.

\subsection{Effective Chebotarev density} \label{effectivechebotarevdensitysection}

An effective version of the Chebotarev density theorem was first proved by Lagarias and Odlyzko \cite{lagariasodlyzko} and further refined by Serre \cite{serre}.  We will now describe the theorem precisely in the form we will use it.

The Chebotarev density theorem gives an asymptotic formula for the number of primes $p \leq x$ for which the associated Frobenius automorphism has a prescribed action on a given fixed number field.  More precisely, let $K$ be a number field which is Galois over $\mbq$ with Galois group $G := \gal(K/\mbq)$ and discriminant $d_K$.  Furthermore, fix any subset $\mc{C} \subseteq G$ satisfying 
\begin{equation} \label{closedunderconjugation}
\forall \gs \in G, \quad \gs \mc{C} \gs^{-1} = \mc{C}.
\end{equation}
For any rational prime $p$ which doesn't divide $d_K$, let $\frob_p \subseteq G$ denote the conjugacy class in $G$ of the Frobenius automorphism $\frob_{\mf{P}}$ attached to any prime ideal $\mf{P} \subseteq \mc{O}_K$ lying over $p\mbz$.  By \eqref{closedunderconjugation}, either $\frob_p \subseteq \mc{C}$ or $\frob_p \cap \mc{C} = \emptyset$, and we consider the counting function
\begin{equation*} 
\pi(x; K/\mbq, \mc{C}) := | \{ p \leq x ; \; p \nmid d_K \, \text{ and } \, \frob_p \subseteq \mc{C} \} |.
\end{equation*}
The Chebotarev density theorem asserts that, as $x \longrightarrow \infty$, one has
\[
\pi(x; K/\mbq, \mc{C}) \sim \frac{| \mc{C} | }{| G |} \pi(x).
\]
We will require the following effective version, which bounds the error term in this asymptotic in terms of data attached to the number field $K$.
\begin{Theorem} (Effective Chebotarev Theorem) \label{effectivechebotarevthm}
There exist absolute positive constants $c_1$ and $c_2$ (with $c_1$ effective) such that, if $x \geq 2$ and 
\begin{equation} \label{xbigenough}
\sqrt{\frac{\log x}{[K : \mbq ]}} \geq c_2 \max \left\{ \log | d_K |, |d_K|^{1/[K : \mbq ]} \right\},
\end{equation}
then
\[
\pi(x; K/\mbq, \mc{C}) =  \frac{| \mc{C} | }{| G |} \pi(x) + O \left( | \mc{C} | \cdot x \cdot \exp\left( -c_1 \sqrt{\frac{\log x}{[K : \mbq ]}} \right) \right).
\]
\end{Theorem}

\subsection{Bounding each sum in \eqref{initialdecomposition}}

We will now state two propositions which bound respectively the first and second sums occurring on the right-hand side of \eqref{initialdecomposition}.  

First observe that, by the prime number theorem, one has
\[
\sum_{\ell \leq Z} \log \ell \sim Z \quad\quad (Z \longrightarrow \infty),
\]
and consequently there exists a positive real constant $M$ for which
\begin{equation} \label{boundforn}
\forall Z \geq 2, \quad\quad \prod_{\ell \leq Z} \ell \leq e^{MZ}.
\end{equation}
In fact, one can take $M = \log 4$ (see \cite[Theorem 4, p. 11]{tenenbaum}).

\begin{proposition} \label{firstsumboundprop}
Assume that
\[
2 \leq Y \leq Z \leq  \frac{1}{3M + 1} \cdot \log \log x,
\]
where $M$ is as in \eqref{boundforn}.  Then, for $Z$ sufficiently large, one has
\[
\sum_{{\begin{substack} {p \leq x \\ \forall \ell \in [Y, Z), \\  p \nequiv 1 \pmod{\ell}} \end{substack}}} 1 \; \ll \; \frac{\log Y}{\log Z} \cdot \pi(x),
\]
with an absolute implied constant.
\end{proposition}

Our next proposition bounds the second sum in \eqref{initialdecomposition}.  
\begin{proposition} \label{secondsumboundprop}
Suppose that $f : \mbq \longrightarrow \mbq$ is not a global power map.  There exists constants $a_f, b_f > 0$ so that, provided
\begin{equation*} 
a_f \leq Y \leq Z \leq  \left( \frac{ \log x }{ (6c_2 \log \log x )^2} \right)^{1/15},
\end{equation*}
(where $c_2$ is the constant appearing in \eqref{xbigenough}) then one has
\[
\sum_{Y \leq \ell < Z} \sum_{{\begin{substack} {p \in S_f(x) \\ p \equiv 1 \pmod{\ell}} \end{substack}}} 1 \; \ll \; \frac{1}{Y \log Y} \cdot \pi(x) + b_f,
\]
with an absolute implied constant.
\end{proposition}

Inserting the results of Propositions \ref{firstsumboundprop} and \ref{secondsumboundprop} into \eqref{initialdecomposition} and putting 
\[
Y = \frac{\log \log \log x}{(\log \log \log \log x)^2}, \quad\quad Z = \frac{1}{3M+1} \cdot \log \log x,
\]
we see that Theorem \ref{maintheorem} follows.  

The remainder of the paper is devoted to proving the two propositions.  To prove Proposition \ref{firstsumboundprop}, we will apply Theorem \ref{effectivechebotarevthm} with $K$ equal to a cyclotomic field:
\begin{equation*} 
K =  \mbq(\zeta_{n_{Y,Z}}) \quad\quad \left( n_{Y,Z} := \prod_{Y \leq \ell < Z} \ell \right).
\end{equation*}
To prove Proposition \ref{secondsumboundprop}, we will apply the same theorem with $K$ equal to a field extension of the form
\begin{equation*} 
K = \mbq \left( \zeta_\ell, n_1^{1/\ell}, n_2^{1/\ell}, f(n_1)^{1/\ell}, f(n_2)^{1/\ell} \right),
\end{equation*} 
for appropriately chosen $n_1, n_2 \in \mbn$.

\subsection{Cyclotomic extensions} \label{cyclotomicextensionssection}

We will now state a few preparatory lemmas about the discriminant and Frobenius automorphism in cyclotomic fields.

\begin{lemma} \label{cyclotomicdiscboundlemma}
Let $n \geq 1$ be a positive integer and let $K = \mbq(\zeta_n)$.  The discriminant $d_K$ is given by
\[
d_K = (-1)^{\varphi(n)/2} \cdot \frac{n^{\varphi(n)}}{\prod_{p \mid n} p^{\varphi(n)/(p-1)}}.
\]
In particular
\[
| d_K | \leq n^{\varphi(n)},
\]
and a prime number $p$ is ramified in $\mbq(\zeta_n)$ if and only if $p$ divides $n$.
\end{lemma}
\begin{proof}
This is classical; see for instance \cite[Proposition 2.7, p. 12]{washington}.  
\end{proof}

Any prime $p$ not dividing $d_K$ is unramified in $K$, and given a prime ideal $\mf{P} \subseteq \mc{O}_K$ lying above $p\mbz \subseteq \mbz$, we may consider the Frobenius automorphism at $\mf{P}$ in $\gal(K/\mbq)$, which we denote by
\[
\frob_\mf{P} \in \gal(K/\mbq).
\]
When $K$ is abelian over $\mbq$, the automorphism $\frob_\mf{P}$ is independent of the choice of $\mf{P}$ over $p$.  We will thus denote it by $\frob_p$ in this case, since it depends only on $p$.  Furthermore, when $K = \mbq(\zeta_n)$, one has the following result, which identifies $\frob_p \in \gal(\mbq(\zeta_n)/\mbq)$, under the canonical group isomorphism
\begin{equation} \label{galQzetanoverQ}
\begin{split}
\gal(\mbq(\zeta_n)/\mbq) &\leftrightarrow (\mbz/n\mbz)^\times \\
 (\zeta_n \mapsto \zeta_n^a) & \mapsto a.
\end{split}
\end{equation}
\begin{lemma} \label{cyclotomicfroblemma}
If $p$ does not divide $n$, then $p$ is unramified in $\mbq(\zeta_n)$.  Furthermore, under the isomorphism \eqref{galQzetanoverQ}, the Frobenius automorphism $\frob_p \in \gal(\mbq(\zeta_n)/\mbq)$ is identified with $p \in (\mbz/n\mbz)^\times$.
\end{lemma}
\begin{proof}
See for instance \cite[Lemma 2.12]{washington}, and the discussion thereafter.
\end{proof}

\subsection{Proof of Proposition \ref{firstsumboundprop}} \label{proofoffirstpropositionsection}

We are now ready to prove Proposition \ref{firstsumboundprop}.  Notice that, by Lemmas \ref{cyclotomicdiscboundlemma} and  \ref{cyclotomicfroblemma}, one has that for any prime $p$,
\[
\forall \ell \in [Y, Z), \,   p \bnequiv 1 \pmod{\ell} 
\; \Longrightarrow \; 
p \in [Y, Z) \, \text{ or } \,
\frob_p \in \mc{C},
\]
where, under $\gal(\mbq(\zeta_n)/\mbq) \simeq (\mbz/n\mbz)^\times$ and the isomorphism 
\[
(\mbz/n_{Y,Z}\mbz)^\times \simeq \prod_{Y \leq \ell < Z} (\mbz/\ell\mbz)^\times
\]
of the Chinese remainder theorem, 
\begin{equation*} 
\mc{C} := \prod_{Y \leq \ell < Z} \left( (\mbz/\ell\mbz)^\times - \{ 1 \} \right).
\end{equation*}
Thus we have
\begin{equation} \label{cyclotomiccalc1}
\sum_{{\begin{substack} {p \leq x \\ \forall \ell \in [Y, Z), \\  p \nequiv 1 \pmod{\ell}} \end{substack}}} 1 
\; \leq \; 
\pi(x; K/\mbq, \mc{C}) + \pi(Z).
\end{equation}

We will now apply Theorem \ref{effectivechebotarevthm} to bound $\pi(x; K/\mbq, \mc{C})$.  We begin by using Lemma \ref{cyclotomicdiscboundlemma} to establish a bound for $Z$ sufficient to guarantee that condition \eqref{xbigenough} is satisfied.

Note that
\[
[K : \mbq ] = \varphi(n_{Y,Z}) \leq n_{Y,Z} \leq e^{MZ},
\]
by \eqref{boundforn}.  Thus,
\[
| d_K |^{1/[K : \mbq ]} \leq (n_{Y,Z}^{\varphi(n_{Y,Z})})^{1/\varphi(n_{Y,Z})} \leq e^{MZ},
\]
and also
\[
\log | d_K | \leq \varphi(n_{Y,Z}) \log n_{Y,Z} \leq e^{MZ} \cdot (MZ).
\]
We therefore have the following corollary of Lemma \ref{cyclotomicdiscboundlemma}.
\begin{cor}
For $K =  \mbq(\zeta_{n_{Y,Z}})$ with $\ds n_{Y,Z} := \prod_{Y \leq \ell < Z} \ell$, one has
\[
\max \left\{ \log | d_K |, |d_K|^{1/[K : \mbq ]} \right\} \leq MZ e^{MZ}.
\]
In particular, if
\begin{equation} \label{boundforZintermsofx}
Z \leq \frac{1}{3M + 1} \log \log x,
\end{equation}
then for $x$ large enough, \eqref{xbigenough} is satisfied in this case.
\end{cor}

Returning to \eqref{cyclotomiccalc1}, note that
\[
| \mc{C} | = \prod_{Y \leq \ell < Z} (\ell - 2) = [K : \mbq ] \cdot \prod_{Y \leq \ell < Z} \left( 1 - \frac{1}{\ell - 1} \right) \ll [K : \mbq ] \cdot \frac{\log Y}{\log Z},
\]
by Merten's theorem.  Furthermore, \eqref{boundforZintermsofx} implies that
\[
e^{MZ} < (\log x)^{1/3}.
\]
Thus, assuming \eqref{boundforZintermsofx}, Theorem \ref{effectivechebotarevthm} implies 
\begin{equation} \label{cyclotomiccalc2}
\begin{split}
\pi(x; K/\mbq, \mc{C}) 
& = \frac{| \mc{C} |}{ [K : \mbq ]} \cdot \pi(x) + O \left( | \mc{C} | \cdot x \cdot \exp\left( -c_1 \sqrt{\frac{\log x}{[K : \mbq ]}} \right) \right) \\
& \ll  \frac{\log Y}{\log Z} \cdot \pi(x) +  e^{MZ} \cdot x \cdot \exp\left( -c_1 \sqrt{\frac{\log x}{e^{Mz}}} \right) \\
& \leq \frac{\log Y}{\log Z} \cdot \pi(x) +  (\log x)^{1/3} \cdot x \cdot \exp\left( -c_1 (\log x)^{1/3} \right).
\end{split}
\end{equation}
For any $A > 0$ one has 
\begin{equation} \label{expboundpoweroflog}
\exp\left( -c_1 (\log x)^{1/3} \right) \ll_A \frac{1}{(\log x)^A},
\end{equation}
so by inserting \eqref{cyclotomiccalc2} into \eqref{cyclotomiccalc1} we conclude that
\[
\sum_{{\begin{substack} {p \leq x \\ \forall \ell \in [Y, Z), \\  p \nequiv 1 \pmod{\ell}} \end{substack}}} 1 \ll \frac{\log Y}{\log Z} \cdot \pi(x) +  \frac{x}{(\log x)^2} + \frac{Z}{\log Z}.
\]
In light of \eqref{boundforZintermsofx} and the prime number theorem, we have
\[
\frac{Z}{\log Z} \ll \log \log x \ll \frac{x}{(\log x)^2} \ll \frac{x}{ \log x \log \log \log x} \ll \frac{\log Y}{\log Z} \cdot \pi(x)
\]
and so this finishes the proof of Proposition \ref{firstsumboundprop}.

\subsection{Kummer extensions} \label{kummerextensionssection}

To prove Proposition \ref{secondsumboundprop}, we will apply Theorem \ref{effectivechebotarevthm} to a field extension of the form 
\begin{equation*} 
K = \mbq \left( \zeta_\ell, n_1^{1/\ell}, n_2^{1/\ell}, f(n_1)^{1/\ell}, f(n_2)^{1/\ell} \right),
\end{equation*} 
for appropriately chosen $n_1, n_2 \in \mbn$.  
In order to do this, we need some control on the Galois group $\gal(K/\mbq)$ in this case.  If $f$ is a global power map then $K = \mbq(\zeta_\ell, n_1^{1/\ell}, n_2^{1/\ell})$, and one cannot deduce the result of Proposition \ref{secondsumboundprop}.  In case $f$ is not a global power map but nevertheless $| S_f | = \infty$, then it is still not immediately clear that one may find $n_1, n_2 \in \mbn$ for which $[ K : \mbq(\zeta_\ell) ] = \ell^4$ for all primes $\ell$ which are large enough, but we show that one may achieve $[ K : \mbq(\zeta_\ell) ] \geq \ell^3$ for $\ell \gg_f 1$, which suffices for our purposes (see Corollary \ref{lowerdegreeboundcor} below).

We begin by reviewing some fundamental facts about Kummer extensions.  For any integers $m \geq 0$ and $n \geq 1$ and vector ${\bf{c}} = (c_1, c_2, \dots, c_m) \in (\mbq^\times)^m$, we will call a number field of the form
\[
K = \mbq(\zeta_n, {\bf{c}}^{1/n}) := 
\mbq(\zeta_n, c_1^{1/n}, c_2^{1/n}, \dots, c_m^{1/n})
\]
a \emph{Kummer extension} (in case $m = 0$, we interpret this as $\mbq(\zeta_n, {\bf{c}}^{1/n}) := \mbq(\zeta_n)$).  In our application, we will deal exclusively with the case where $n = \ell$ is an odd prime number, and we begin by describing the associated Galois group.  Consider the group
\[
(\mbz/\ell\mbz)^\times \ltimes (\mbz/\ell\mbz)^m,
\]
where the semi-direct product is defined via the multiplicative action of $(\mbz/\ell\mbz)^\times$ on $(\mbz/\ell\mbz)^m$, or explicitly 
\[
(a_1,{\bf{b}}_1) \cdot (a_2, {\bf{b}}_2) = (a_1a_2, {\bf{b}}_2 + a_2{\bf{b}}_1),
\]
where ${\bf{b}}_i \in (\mbz/\ell\mbz)^m$.  (Equivalently, the embedding
\begin{equation*}  
\begin{split}
(\mbz/\ell\mbz)^\times \ltimes (\mbz/\ell\mbz)^m &\hookrightarrow GL_{m+1}(\mbz/\ell\mbz) \\
(a, {\bf{b}}) &\mapsto 
\begin{pmatrix}
a & 0 \\
{\bf{b}} & I
\end{pmatrix},
\end{split}
\end{equation*}
where $I$ denotes the $m \times m$ identity matrix, allows one to regard $(\mbz/\ell\mbz)^\times \ltimes (\mbz/\ell\mbz)^m$ as a subgroup of $GL_{m+1}(\mbz/\ell\mbz)$.)  
There is an embedding of groups\footnote{Here we are interpreting $\gal(\mbq(\zeta_\ell, {\bf{c}}^{1/\ell})/\mbq)$ as operating on the \emph{right}.}
\begin{equation} \label{kummerembedding}
\begin{split}
\gal(\mbq(\zeta_\ell, {\bf{c}}^{1/\ell})/\mbq) &\hookrightarrow  (\mbz/\ell\mbz)^\times \ltimes (\mbz/\ell\mbz)^m \\
\begin{pmatrix}
\zeta_\ell & \mapsto & \zeta_\ell^a \\
c_i^{1/\ell} & \mapsto & c_i^{1/\ell} \cdot \zeta_\ell^{b_i}
\end{pmatrix}
&\mapsto
(a, {\bf{b}}),
\end{split}
\end{equation}
where  ${\bf{b}} = (b_1, b_2, \dots b_m)$. 
What is the image of this embedding?  In general, the image depends on whether (and to what extent) there exist multiplicative relations
\begin{equation} \label{relationsforc}
{\bf{c}}^{{\bf{d}}/\ell} := \prod_{i = 1}^m (c_i^{1/\ell})^{d_i} \in \mbq^\times,
\end{equation}
where in the above, ${\bf{d}} = (d_1, d_2, \dots, d_m) \in (\mbz/\ell\mbz)^m$.  In our application, we will need to understand the image of this embedding, even in the case where nontrivial relations such as \eqref{relationsforc} exist.

Let $V_{\bf{c}}(\ell)$, respectively $V_{\bf{c}}^\perp(\ell)$ denote the $\mbz/\ell\mbz$-vector subspaces
\begin{equation} \label{defofV}
\begin{split}
V_{\bf{c}}(\ell) &:= \{ {\bf{d}} \in (\mbz/\ell\mbz)^m : \; \text{ the relation \eqref{relationsforc} holds} \} \\
V_{\bf{c}}^\perp(\ell) &:= \{ {\bf{b}} \in (\mbz/\ell\mbz)^m : \; \forall {\bf{d}} \in V_{\bf{c}}(\ell), \, \sum_{i=1}^m b_i d_i \equiv 0 \pmod{\ell} \}.
\end{split}
\end{equation}
It follows from \eqref{relationsforc} and \eqref{defofV} that the image of the embedding \eqref{kummerembedding} is equal to the subgroup
\[
(\mbz/\ell\mbz)^\times \ltimes V_{\bf{c}}^\perp(\ell) \subseteq (\mbz/\ell\mbz)^\times \ltimes (\mbz/\ell\mbz)^m.
\]
The following lemma summarizes our discussion, and uses the notation
\begin{equation} \label{defofd}
d_{\bf{c}} := \dim_{\mbz/\ell\mbz} V_{\bf{c}}^\perp(\ell) \in \mbn \cup \{ 0 \}.
\end{equation}
\begin{lemma}
The function \eqref{kummerembedding} gives an isomorphism of groups
\[
\gal(\mbq(\zeta_\ell, {\bf{c}}^{1/\ell})/\mbq) \simeq (\mbz/\ell\mbz)^\times \ltimes V_{\bf{c}}^\perp(\ell).
\]
Furthermore (after possibly re-labelling indices) we have that
\[
\mbq(\zeta_\ell, c_1^{1/\ell}, c_2^{1/\ell}, \dots, c_m^{1/\ell}) = \mbq(\zeta_\ell, c_1^{1/\ell}, c_2^{1/\ell}, \dots, c_d^{1/\ell}),
\]
where $d = d_{\bf{c}}$ is defined by \eqref{defofd}.  This choice of $d$ is the smallest possible\footnote{In case $d = 0$, we make the interpretation $\mbq(\zeta_\ell, c_1^{1/\ell}, c_2^{1/\ell}, \dots, c_d^{1/\ell}) := \mbq(\zeta_\ell)$ and $ (\mbz/\ell\mbz)^\times \ltimes (\mbz/\ell\mbz)^d := (\mbz/\ell\mbz)^\times$.}.  In particular, one has
\begin{equation*} 
\begin{split}
\gal(\mbq(\zeta_\ell, c_1^{1/\ell}, c_2^{1/\ell}, \dots, c_m^{1/\ell})/\mbq) &= \gal(\mbq(\zeta_\ell, c_1^{1/\ell}, c_2^{1/\ell}, \dots, c_d^{1/\ell})/\mbq) \\
&\simeq (\mbz/\ell\mbz)^\times \ltimes (\mbz/\ell\mbz)^d.
\end{split}
\end{equation*}
\end{lemma}
\begin{proof}
Let $B \subseteq \mbq^\times$ be the multiplicative subgroup generated by $(\mbq^\times)^\ell$ and $\{ c_i : \; 1 \leq i \leq m \}$.  In \cite[Theorem 8.1, p. 294--295]{lang} it is shown that 
\[
\gal(\mbq(\zeta_\ell, {\bf{c}}^{1/\ell})/\mbq(\zeta_\ell)) \simeq \frac{B}{(\mbq^\times)^\ell}.
\]
Noting that, under ${\bf{c}}^{{\bf{n}} \pmod{\ell}} \mapsto {\bf{n}} \pmod{\ell}$, one has
\[
\frac{B}{(\mbq^\times)^\ell} \simeq \frac{(\mbz/\ell\mbz)^m}{V_{{\bf{c}}}(\ell)} \simeq V_{{\bf{c}}}^\perp(\ell),
\]
one concludes that $\gal(\mbq(\zeta_\ell, {\bf{c}}^{1/\ell})/\mbq(\zeta_\ell)) \simeq V_{{\bf{c}}}^\perp(\ell)$, and the conclusion of the lemma follows.
\end{proof}

In our proof of Proposition \ref{secondsumboundprop}, it will become important to know that the subspace $V_{\bf{c}}^\perp(\ell) \subseteq (\mbz/\ell\mbz)^m$ is not too small, which motivates the following lemma.  
Let us define the $\mbz$-modules $M_{\bf{c}}$ and $M_{{\bf{c}},\ell}$ by
\begin{equation} \label{defofMc}
\begin{split}
M_{\bf{c}} &:= \left\{ {\bf{n}} \in \mbz^m : \; {\bf{c}}^{{\bf{n}}} \in \{ \pm 1 \} \right\}, \\
M_{{\bf{c}},\ell} &:= \left\{ {\bf{n}} \in \mbz^m : \; {\bf{c}}^{{\bf{n}}} \in ( \mbq^\times )^\ell \right\} \\
&= \left\{ {\bf{n}} \in \mbz^m : \; {\bf{n}} \pmod{\ell} \in V_{{\bf{c}}}(\ell) \right\}.
\end{split}
\end{equation}
Note that, if $\ell$ is an odd prime, then $M_{{\bf{c}}} \subseteq M_{{\bf{c}}, \ell}$.

\begin{lemma} \label{prekummerlargeequivlemma}
Let $m \geq 1$ and ${\bf{c}} \in (\mbq^\times)^m$, and let $\ell$ be an odd prime number.  
Then
\[
[ \mbq( \zeta_\ell, {\bf{c}}^{1/\ell} ) : \mbq(\zeta_\ell) ] < \ell^m 
\; \Longleftrightarrow \; 
M_{{\bf{c}},\ell} \neq \{ {\bf{0}} \}.
\]
\end{lemma}
\begin{proof}
Using the previous lemma and \eqref{defofMc}, one sees that 
\[
[ \mbq( \zeta_\ell, {\bf{c}}^{1/\ell} ) : \mbq(\zeta_\ell) ] < \ell^m \; \Leftrightarrow \; V_{{\bf{c}}}(\ell) \neq \{ {\bf{0}} \} \; \Leftrightarrow \; M_{{\bf{c}},\ell} \neq \{ {\bf{0}} \}.
\]
\end{proof}

Now let $S$ be any set of odd primes.  One concludes from the definitions that 
\begin{equation} \label{obviousversion}
|S| = \infty \; \Longrightarrow \; M_{{\bf{c}}} = \bigcap_{\ell \in S} M_{{\bf{c}},\ell}.
\end{equation}
More is true.
\begin{lemma} \label{nonobviouslemma}
Let $S$ be a set of odd prime numbers.  If $|S| = \infty$ then
\[
M_{{\bf{c}}} \neq \{ {\bf{0}} \} \; \Longleftrightarrow \; \forall \ell \in S, \; M_{{\bf{c}},\ell} \neq \{ {\bf{0}} \}.
\]
\end{lemma}
\begin{proof}
The ``$\Longrightarrow$'' direction is clear from \eqref{obviousversion}.  For the converse, let 
\[
\begin{split}
R := &\{ \text{primes } p : \; v_p(c_i) \neq 0 \text{ for some } i \in \{ 1, 2, \dots, m \} \} \\
= &\{ p_1, p_2, \dots, p_r \},
\end{split}
\]
where $r := |R|$, and define the vectors ${\bf{e}}_j \in \mbz^R$ by 
\[
\begin{split}
c_j =: &\prod_{p \in R} p^{{\bf{e}}_j(p)} \\
= & \prod_{i = 1}^r p_i^{{\bf{e}}_{i,j}}.  
\end{split}
\]
Furthermore, consider the $r \times m$ integer matrix
\[
{\bf{E}}_{{\bf{c}}} :=
\begin{pmatrix}
{\bf{e}}_1 & {\bf{e}}_2 & \dots & {\bf{e}}_m
\end{pmatrix}
\in M_{r \times m}(\mbz)
\]
whose columns are the vectors ${\bf{e}}_j$.
Note that, for any vector ${\bf{n}} \in \mbz^m$, 
\[
\begin{split}
{\bf{c}}^{{\bf{n}}} \in \{ \pm 1 \} \quad &\Longleftrightarrow \quad  {\bf{E}}_{{\bf{c}}} \cdot {\bf{n}} = {\bf{0}}, \\
{\bf{c}}^{{\bf{n}}} \in (\mbq^\times)^\ell \quad &\Longleftrightarrow \quad  {\bf{E}}_{{\bf{c}}} \cdot {\bf{n}} \equiv {\bf{0}} \pmod{\ell},
\end{split}
\]
so in particular 
\[
\begin{split}
M_{{\bf{c}}} =& \{ {\bf{n}} \in \mbz^m : \; {\bf{E}}_{{\bf{c}}} \cdot {\bf{n}} = {\bf{0}} \}, \quad \text{ and} \\
M_{{\bf{c}},\ell} =& \{ {\bf{n}} \in \mbz^m : \; {\bf{E}}_{{\bf{c}}} \cdot {\bf{n}} \equiv {\bf{0}} \pmod{\ell} \}.
\end{split}
\]
Note that, if $r < m$ then necessarily $\ker {\bf{E}}_{{\bf{c}}}$ has dimension at least one, so $M_{{\bf{c}}} \neq \{ {\bf{0}} \}$ in this case.  In case $r \geq m$, let $\ds w := \binom{r}{m} \in \mbn$ and let $\gD_{{\bf{c}}} \in \mbz^w$ be the vector of determinants of all $m \times m$ sub-matrices of ${\bf{E}}_{{\bf{c}}}$.  One has
\begin{equation} \label{gDequivalences}
\begin{split}
M_{{\bf{c}}} \neq \{ {\bf{0}} \} \; &\Longleftrightarrow \; \gD_{\bf{c}} = {\bf{0}}, \\
M_{{\bf{c}},\ell} \neq \{ {\bf{0}} \} \; &\Longleftrightarrow \; \gD_{\bf{c}} \equiv {\bf{0}} \pmod{\ell}.
\end{split}
\end{equation}
Thus,
\[
\forall \ell \in S, \; M_{{\bf{c}},\ell} \neq \{ {\bf{0}} \} \; \Rightarrow \; \forall \ell \in S, \; \gD_{\bf{c}} \equiv {\bf{0}} \pmod{\ell} \; \Rightarrow \; \gD_{\bf{c}} = {\bf{0}} \; \Rightarrow \; M_{\bf{c}} \neq \{ {\bf{0}} \},
\]
proving the lemma.
\end{proof}

The next lemma will be useful for making sure that our Kummer extensions are not too small.
\begin{lemma} \label{kummerlargeequivlemma}
Let $m \geq 1$ and ${\bf{c}} \in (\mbq^\times)^m$, and let $S$ be an infinite set of odd prime numbers.
One has
\[
\forall \ell \in S, \; \left[ \mbq \left(\zeta_\ell, {\bf{c}}^{1/\ell} \right) : \mbq( \zeta_\ell) \right] < \ell^m \; \Longleftrightarrow \; 
M_{{\bf{c}}} \neq \{ {\bf{0}} \}.
\]
\end{lemma}
\begin{proof}
Apply Lemmas \ref{prekummerlargeequivlemma} and \ref{nonobviouslemma}.
\end{proof}

\begin{remark} \label{kummerboundremark}
Suppose that ${\bf{c}} \in (\mbq^\times)^m$ and $M_{{\bf{c}}} = \{ {\bf{0}} \}$.  By \eqref{gDequivalences}, we have that $\gD_{\bf{c}} \neq {\bf{0}}$.  Thus, writing $\gD_{{\bf{c}}} =: (\gl_1, \gl_2, \dots, \gl_w)$, we may define $\gd({\bf{c}}) := 2 \gcd( \gl_1, \gl_2, \dots, \gl_w)$. 
The proof of Lemma \ref{nonobviouslemma} shows that
\[
\ell \nmid \gd({\bf{c}}) \; \Longrightarrow \; \left[ \mbq \left(\zeta_\ell, {\bf{c}}^{1/\ell} \right) : \mbq(\zeta_\ell) \right] = \ell^m.
\]
\end{remark}

The next corollary follows from applying Lemma \ref{kummerlargeequivlemma} with ${\bf{c}} = (n_1, n_2, f(n_1)) \in (\mbq^\times)^3$ with $n_1, n_2 \in \mbn$ chosen in accordance with Corollary \ref{threenumberslemma}.  Let us make the following definitions, for ${\bf{n}} = (n_1, n_2) \in \mbn^2$:
\begin{equation*} 
\begin{split}
{\bf{c}}_{f, {\bf{n}}} :=& (n_1, n_2, f(n_1)) \in (\mbq^\times)^3 \\
\mc{N}_f :=& \{ {\bf{n}} \in \mbn^2 : \; M_{{\bf{c}}_{f,{\bf{n}}}} = \{ {\bf{0}} \} \}.
\end{split}
\end{equation*}
If ${\bf{n}} \in \mc{N}_f$, then the vector $\gD_{{\bf{c}}_{f, {\bf{n}}}} =: (\gl_1^{(f,{\bf{n}})}, \gl_2^{(f,{\bf{n}})}, \dots, \gl_w^{(f,{\bf{n}})}) \in \mbz^w$ appearing in the proof of Lemma \ref{nonobviouslemma} is well-defined and non-zero.  We then set
\begin{equation*} 
\gd_{f, {\bf{n}}} := 2 \gcd(\gl_1^{(f,{\bf{n}})}, \gl_2^{(f,{\bf{n}})}, \dots, \gl_w^{(f,{\bf{n}})}).
\end{equation*}

\begin{cor} \label{lowerdegreeboundcor}
Suppose that $f : \mbq \longrightarrow \mbq$ satisfies $| S_f | = \infty$.  Then either $f$ is a global power map, or $\mc{N}_f \neq \emptyset$.  Furthermore, for any ${\bf{n}} = (n_1, n_2) \in \mc{N}_f$ and for any odd prime $\ell$ one has
\begin{equation*} 
\ell \nmid \gd_{f,{\bf{n}}} \; \Longrightarrow \; \left[ \mbq \left( \zeta_\ell, n_1^{1/\ell}, n_2^{1/\ell}, f(n_1)^{1/\ell}, f(n_2)^{1/\ell} \right) : \mbq(\zeta_\ell) \right] \geq \ell^3
\end{equation*}
\end{cor}
\begin{proof}
Let $n \in \mbn$.  By Lemma \ref{kummerlargeequivlemma}, $\left[ \mbq \left( \zeta_\ell, n^{1/\ell}, f(n)^{1/\ell} \right) : \mbq(\zeta_\ell) \right] < \ell^2$ for infinitely many primes $\ell$ if and only if 
\begin{equation} \label{relationwithnandfofn}
n^c f(n)^d \in \{ \pm 1 \} \; \text{ for some } \; (c,d) \in \mbz^2 - \{ (0,0) \},
\end{equation}
and we may as well take $c$ and $d$ to be relatively prime.  If $n$ is further assumed to be square-free and greater than $1$, then one finds that $d = 1$ in \eqref{relationwithnandfofn}, and so this happens if and only if
\begin{equation} \label{equivrelationwithnandfofn}
f(n) \in n^\mbz \cup -n^\mbz.
\end{equation}
By Corollary \ref{threenumberslemma}, one may find a square-free number $n$ for which \eqref{equivrelationwithnandfofn} does not happen.  Putting $n_1 := n$ and taking $n_2 = p$ to be any prime for which $v_p(n_1) = v_p(f(n_1)) = 0$, we see that $M_{{\bf{c}}_{f, {\bf{n}}}} = \{ {\bf{0}} \}$.  Applying Remark \ref{kummerboundremark}, we see that, for $\ell \nmid \gd_{f, {\bf{n}}}$, one has$[ \mbq(\zeta_\ell, {\bf{c}}_{f, {\bf{n}}}^{1/\ell}) : \mbq(\zeta_\ell) ] = \ell^3$, which proves the corollary.
\end{proof}

The next lemma deals with the absolute discriminant of the field 
\begin{equation} \label{Kkummergeneral}
\begin{split}
K = \mbq(\zeta_\ell, {\bf{c}}^{1/\ell}) :=& \mbq(\zeta_\ell, c_1^{1/\ell}, c_2^{1/\ell}, \dots, c_m^{1/\ell}) \\
=& \mbq(\zeta_\ell, c_1^{1/\ell}, c_2^{1/\ell}, \dots, c_d^{1/\ell}),
\end{split}
\end{equation}
where $0 \leq d = d_{{\bf{c}}} \leq m$ is as in \eqref{defofd}.
Its proof utilizes the following classical formula for relative discriminants.  
\begin{lemma} \label{relativediscriminantformulalemma}
Let $F \subseteq L \subseteq K$ be a tower of number fields, let $\gD_{K/L} \subseteq \mc{O}_L$, $\gD_{K/F} \subseteq \mc{O}_F$, and $\gD_{L/F} \subseteq \mc{O}_F$ be the relative discriminants and let $\mf{N}_{L/F} : L^\times \longrightarrow F^\times$ the usual norm map.  Then one has
\begin{equation} \label{relativediscriminantformula}
\gD_{K/F} = \mf{N}_{L/F}(\gD_{K/L}) \gD_{L/F}^{[ K : L ]}.
\end{equation}
\end{lemma}
\begin{proof}
See for instance \cite[p. 126]{frolichtaylor}.
\end{proof}

\begin{lemma} \label{kummerdisclemma}
Let $K$ be as in \eqref{Kkummergeneral}.  Then the absolute discriminant $d_K$ divides 
\[
\left( \prod_{i=1}^d \num(c_i) \den(c_i) \right)^{(\ell - 1)^2\ell^{d-1}} \ell^{(d+1)(\ell-1)^2\ell^{d-1}}.
\]
\end{lemma}
\begin{proof}
We induct on $d$.   We will apply Lemma \ref{relativediscriminantformulalemma} with $F = \mbq$,
\[
L = 
\begin{cases}
\mbq(\zeta_\ell) & \text{ if } d = 1 \\
\mbq(\zeta_\ell, c_1^{1/\ell}, c_2^{1/\ell}, \dots, c_{d-1}^{1/\ell}) & \text{ if } d > 1
\end{cases}
\]
and $K = L(c_d^{1/\ell})$.  First note that, for any $\ga \in \mc{O}_K$ satisfying $K = L(\ga)$, one has
\begin{equation} \label{discriminantcontainment}
\mf{D}(\ga) \cdot \mc{O}_L \subseteq \gD_{K/L},
\end{equation}
where $\mf{D}(\ga)$ is the square of the determinant of the $\ell \times \ell$ matrix whose $(i,j)$-th entry is $\gs_i(\ga^j)$, where $\{ \gs_0, \gs_1, \dots, \gs_{\ell-1} \}$ is the set of embeddings of $K$ into $\mbc$ fixing $L$ point-wise and $0 \leq j \leq \ell-1$.  Let us abbreviate $c := c_d$.  Writing $c = a/b$ in lowest terms, and applying \eqref{discriminantcontainment} with $\ga = (a b^{\ell - 1})^{1/\ell} = c^{1/\ell} b$ and again with $\ga = (a^{\ell - 1} b)^{1/\ell} = c^{-1/\ell} a$, we conclude that
\[
\mf{D}\left( (ab^{\ell - 1})^{1/\ell} \right) \cdot \mc{O}_L + \mf{D}\left( (a^{\ell - 1}b)^{1/\ell} \right) \cdot \mc{O}_L \subseteq \gD_{K/L}.
\]
Now for any integer $n$, one computes that $\mf{D}(n^{1/\ell}) = n^{\ell-1} \ell^{\ell-2}$.  Using this, the greatest common divisor on the left-hand side is readily calculated, showing that
\[
(\num(c) \den(c))^{\ell-1} \ell^{\ell-2} \mc{O}_L \subseteq \gD_{K/L}.
\]
Inserting this information into \eqref{relativediscriminantformula}, we find that
\[
\gD_{K/\mbq} \quad \text{ divides } \quad (\num(c) \den(c))^{(\ell-1)^2 \ell^{d-1}} \ell^{(\ell-2)(\ell-1)\ell^{d-1}} \gD_{L/\mbq}^\ell.
\]
Applying the induction hypothesis (or the formula $d_{\mbq(\zeta_\ell)} = \pm \ell^{\ell-2}$ of Lemma \ref{cyclotomicdiscboundlemma} in the base case), the conclusion of Lemma \ref{kummerdisclemma} now follows.
\end{proof}

In particular, since $\ds |d_K | \leq \left( \prod_{i=1}^d \num(c_i) \den(c_i) \right)^{(\ell-1)^2\ell^{d-1}} \ell^{(d+1)(\ell-1)^2\ell^{d-1}}$, we obtain the following corollary.  Let us put 
\begin{equation*} 
b_{f, {\bf{n}}} :=  \left| \prod_{i=1}^2 n_i \num(f(n_i)) \den(f(n_i)) \right|.
\end{equation*}

\begin{cor} \label{kummerdiscboundcor}
Suppose $f : \mbq \longrightarrow \mbq$ is any function and let $K = \mbq \left( \zeta_\ell, n_1^{1/\ell}, n_2^{1/\ell}, f(n_1)^{1/\ell}, f(n_2)^{1/\ell} \right)$.  Then for any prime $\ell$ satisfying $[K : \mbq(\zeta_\ell) ] \geq \ell^3$ and $\log \ell \geq b_{f, {\bf{n}}}$, one has
\[
\max \left\{ \log | d_K |, |d_K|^{1/[K : \mbq ]} \right\} \; \leq \; 6 \ell^5 \log \ell.
\]
\end{cor}

\subsection{The Frobenius automorphism in Kummer extensions}
We now turn our consideration to the Frobenius automorphism $\frob_\mf{P}$ for a prime ideal $\mf{P} \subseteq \mc{O}_K$ lying over $p\mbz$, where $K = \mbq(\zeta_\ell, {\bf{c}}^{1/\ell})$ and 
$p \equiv 1 \pmod{\ell}$.

We begin by describing the situation when $m = d = 1$, i.e. (dropping subscripts) we have
\[
K = K_c := \mbq(\zeta_\ell, c^{1/\ell}) \neq \mbq(\zeta_\ell) \quad\quad (c \in \mbq^\times)
\]
and
\begin{equation} \label{kummerisomorphism}
\begin{split}
\gal(K_c/\mbq) &\simeq (\mbz/\ell\mbz)^\times \ltimes \mbz/\ell\mbz \\
\begin{pmatrix}
\zeta_\ell & \mapsto & \zeta_\ell^a \\
c^{1/\ell} & \mapsto & c^{1/\ell} \cdot \zeta_\ell^b
\end{pmatrix}
&\mapsto
(a, b),
\end{split}
\end{equation}
The minimal polynomials over $\mbq$ of $\zeta_\ell$ and $c^{1/\ell}$, together with there factorizations over $\ol{\mbq}$, are given respectively as follows:
\begin{equation*}
\begin{split}
\Phi_\ell(t) := \frac{t^\ell - 1}{t - 1} &= \prod_{i \in (\mbz/\ell\mbz)^\times} (t - \zeta_\ell^i), \\
t^\ell - c &= \prod_{i \in \mbz/\ell\mbz} (t - \zeta_\ell^i \cdot c^{1/\ell}).
\end{split}
\end{equation*}

In our present discussion, we will adopt the standing assumptions that 
\begin{equation} \label{assumptionsonp}
p \equiv 1 \pmod{\ell} \quad \text{ and } \quad \ord_p(c) = 0.
\end{equation}
By Lemmas \ref{cyclotomicfroblemma} and \ref{kummerdisclemma}, these conditions imply that
\[
\text{$p$ splits completely in $\mbq(\zeta_\ell)$ } \; \text{ and } \; \text{ $p$ is unramified in $K_c$.}
\]
Consider the subgroup $\mu_\ell \subseteq \ol{\mbf_p}^\times$ of $\ell$-th roots of unity.
Since $p \equiv 1 \pmod{\ell}$, one can find an element $z \in \mbz$ whose reduction $\hat{z}$ modulo $p$ generates $\mu_\ell$, i.e. we have
\begin{equation} \label{defofz}
\langle \hat{z} \rangle = \mu_\ell \subseteq \mbf_p^\times,
\end{equation}
and the reductions modulo $p$ of the above minimal polynomials factorize over $\ol{\mbf_p}$ as
\[
\begin{split}
\Phi_\ell(t) &\equiv \prod_{i \in (\mbz/\ell\mbz)^\times} (t - \hat{z}^i) \pmod{p}, \\
t^\ell - c &\equiv \prod_{i \in \mbz/\ell\mbz} (t - \hat{z}^i \theta_c) \pmod{p},
\end{split}
\]
for some $\theta_c \in \ol{\mbf_p}^\times / \mu_\ell$.  Furthermore, one has the prime factorization
\begin{equation*} 
p \mc{O}_{\mbq(\zeta_\ell)} = \prod_{i \in (\mbz/\ell\mbz)^\times} \mf{p}_{z,i}, \quad\quad \left( \mf{p}_{z,i} := p \mc{O}_{\mbq(\zeta_\ell)} + (\zeta_\ell - z^{i^*}) \mc{O}_{\mbq(\zeta_\ell)} \right),
\end{equation*}
where $i^*$ denotes an integer satisfying $i^* i \equiv 1 \pmod{p}$.
Note that, by our choice of indexing, we have
\begin{equation} \label{transformrule}
\forall j \in (\mbz/\ell\mbz)^\times, \quad\quad \mf{p}_{z,i} = \mf{p}_{z^j,ij}.
\end{equation}
What about the splitting type of such a prime ideal $\mf{p}_{z,i}$ in $K_c$?  Since $K_c$ has prime degree $\ell$ over $\mbq(\zeta_\ell)$ and by \eqref{assumptionsonp}, each $\mf{p}_{z,i}$ either splits completely or remains inert in $K_c$.  Furthermore, since $K_c$ is Galois over $\mbq$, the splitting type of each $\mf{p}_{z,i}$ is the same.  Under the assumptions \eqref{assumptionsonp}, one has
\begin{equation} \label{splittingcompletelycriterionforc}
\begin{split}
\text{$\mf{p}_{z,i}$ splits completely in $K_c$ }
\; &\Longleftrightarrow \;
c \pmod{p} \in (\mbf_p^\times)^\ell \\
\; &\Longleftrightarrow \;
\theta_c \in \mbf_p^\times / \mu_\ell
\end{split}
\end{equation}
If this is the case, we may allow $z$ in \eqref{defofz} to be an arbitrary generator of $\mu_\ell \subset \mbf_p^\times$ and note also that, under the isomorphism \eqref{kummerisomorphism},
\[
\text{$\mf{p}_{z,i}$ splits completely in $K_c$ }
\; \Longleftrightarrow \; 
\frob_\mf{P} = (1, 0),
\]
for any prime ideal $\mf{P} \subseteq \mc{O}_{K_c}$ lying over $\mf{p}_{z,i}$.  

In case $\mf{p}_{z,i}$ does not split completely in $K_c$, the finite field $\mbf_p[\theta_c]$ has degree $\ell$ over $\mbf_p$, and we normalize our choice of $z = z_c \in \mbz$ so that
\begin{equation} \label{normalizationofzc}
\hat{z}_c := \frac{\theta_c^p}{\theta_c} \in \mbf_p^\times
\end{equation}
(Note that $\hat{z}_c \in \mbf_p^\times$ is independent of the choice of $\theta_c \in \ol{\mbf_p}^\times / \mu_\ell$).  In this case, putting
\begin{equation} \label{defofmfPi}
\mf{P}_{z_c,i} := \mf{p}_{z_c,i} \mc{O}_{K_c} = p \mc{O}_{K_c} + (\zeta_\ell - z_c^{i^*}) \mc{O}_{K_c},
\end{equation}
the ideal $\mf{P}_{z_c,i}$ is prime and we have a prime factorization 
\[
p \mc{O}_{K_c} = \prod_{i \in (\mbz/\ell\mbz)^\times} \mf{P}_{z_c,i}.  
\]
The following lemma characterizes the Frobenius automorphism $\frob_{\mf{P}_{z_c,i}}$.
\begin{lemma} \label{frobinkummerlemma}
Suppose $c \in \mbq^\times$ satisfies $K_c := \mbq(\zeta_\ell, c^{1/\ell}) \neq \mbq(\zeta_\ell)$.  Furthermore, let $p$ be a prime number satisfying $p \equiv 1 \pmod{\ell}$ and $\ord_p(c) = 0$.  Then $p$ is unramified in $K_c$ and, with notation as above, under the isomorphism $\gal(K_c/\mbq) \simeq (\mbz/\ell\mbz)^\times \ltimes \mbz/\ell\mbz$, one has
\[
\frob_{\mf{P}_{z_c,i}} = 
\begin{cases}
(1, 0) & \text{ if $p$ splits completely in $K_c$ and $\mf{P}_{z_c,i}$ is any prime above $p$}\\
(1, i ) & \text{ if $p\mc{O}_{K_c} = \prod_{i \in (\mbz/\ell\mbz)^\times} \mf{P}_{z_c,i}$, where each $\mf{P}_{z_c,i}$ is as in \eqref{defofmfPi} and is prime.}
\end{cases}
\]
\end{lemma}
\begin{proof}
We need only concern ourselves with the case that $p$ does not split completely in $K_c$.  In this case, consider the ring homomorphism $\pi_{z_c,i} : \mc{O}_{\mbq(\zeta_\ell)} \longrightarrow \mbf_p$, induced by $\ds \zeta_\ell \mapsto z_c^{i^*} \pmod{p}$.  Note that
\[
\ker \pi_{z_c,i} = \mf{p}_{z_c,i} \quad \text{ and } \quad \gs_{a}( \mf{p}_{z_c,i})  = \mf{p}_{z_c,ai},
\]
where $\gs_a \mapsto a$ under $\ds \gal(\mbq(\zeta_n)/\mbq) \simeq (\mbz/n\mbz)^\times$.  Since $\mc{O}_{K_c}/ \mf{P}_{z_c,i} \simeq \mbf_p(\theta_c)$ in this case, one may extend $\pi_{z_c,i}$ to a ring homomorphism $\ds \varpi_{z_c,i} : \mc{O}_{K_c} \longrightarrow \mbf_p(\theta_c)$ for which $\varpi_{z_c,i}(c^{1/\ell}) = \theta_c$.  Consider the induced isomorphism
\[
\varpi_{z_c,i} : \mc{O}_{K_c} / \mf{P}_{z_c,i} \longrightarrow \mbf_p(\theta_c).
\]
By definition of $\frob_{\mf{P}_{z_c,i}}$, one has $\varpi_{z_c,i} \circ \frob_{\mf{P}_{z_c,i}} \circ \varpi_{z_c,i}^{-1}(\theta_c) = \theta_c^p$.  On the other hand, if $\frob_{\mf{P}_{z_c,i}} \mapsto (1,b)$ under \eqref{kummerisomorphism}, then by \eqref{normalizationofzc}, we have
\[
z_c \theta_c = \theta_c^p = \varpi_{z_c,i} ( \frob_{\mf{P}_{z_c,i}} ( \varpi_{z_c,i}^{-1}(\theta_c) ))= \varpi_{z_c,i} ( \frob_{\mf{P}_{z_c,i}} ( c^{1/\ell} \pmod{\mf{P}_{z_c,i}} )) = \varpi_{z_c,i}( \zeta_\ell^b c^{1/\ell}) = z_c^{i^*b} \theta_c.
\]
Thus, one finds that $b = i$, proving the lemma.
\end{proof}

The next corollary of Lemma \ref{frobinkummerlemma} is essential in what follows.  We introduce the notation
\begin{equation*} 
K = \mbq(\zeta_\ell, c_1^{1/\ell}, c_2^{1/\ell}, \dots, c_k^{1/\ell}, d_1^{1/\ell}, d_2^{1/\ell}, \dots d_k^{1/\ell}) =: \mbq(\zeta_\ell, {\bf{c}}^{1/\ell}, {\bf{d}}^{1/\ell} ),
\end{equation*}
where ${\bf{c}}, {\bf{d}} \in \mbq^k$ are the obvious vectors, and
\begin{equation} \label{newformofgal}
\begin{split}
\gal(K/\mbq) \hookrightarrow & (\mbz/\ell\mbz)^\times \ltimes (\mbz/\ell\mbz)^{2k} \\
= & (\mbz/\ell\mbz)^\times \ltimes \left( (\mbz/\ell\mbz)^k \times (\mbz/\ell\mbz)^k \right),
\end{split}
\end{equation}
so that elements of $\gal(K/\mbq)$ may be written in the form $(a,{\bf{b}}, {\bf{f}})$ with ${\bf{b}}, {\bf{f}} \in (\mbz/\ell\mbz)^k$.  We denote by $\mc{C}_{2k} \subseteq \gal(K/\mbq)$ the subset
\begin{equation} \label{defofmcC}
\mc{C}_{2k} := \{ (1, {\bf{b}}, {\bf{f}} ) \in \gal(K/\mbq) : \; {\bf{f}} = \gl {\bf{b}} \text{ for some $\gl \in \mbz/\ell\mbz$ } \}.
\end{equation}
Note that, under $\gal(K/\mbq) \hookrightarrow (\mbz/\ell\mbz)^\times \ltimes (\mbz/\ell\mbz)^{2k}$, the subset $\mc{C}_{2k}$ is stable by $\gal(K/\mbq)$-conjugation.
\begin{cor} \label{keycor}
Let $K = \mbq(\zeta_\ell, {\bf{c}}^{1/\ell}, {\bf{d}}^{1/\ell})$ and assume the remaining notation just introduced.  Let $p$ be any prime number satisfying 
\[
\forall i \in \{ 1, 2, \dots, k \}, \quad \ord_p ( c_i) = \ord_p( d_i ) = 0
\]
and $p \equiv 1 \pmod{\ell}$.
Suppose further that, for some fixed $k_p \in \mbz/(p-1)\mbz$, one has
\[
\forall i \in  \{ 1, 2, \dots, k \}, \quad d_i \equiv c_i^{k_p} \pmod{p}.
\]
 Then $p$ is unramified in $K$ and, under the embedding \eqref{newformofgal}, the Frobenius class $\frob_p \subseteq \gal(K/\mbq)$ satisfies
\[
\frob_p \subseteq \mc{C}_{2k}.
\]
\end{cor}
\begin{proof}
Note that, for any vector ${\bf{w}} = (w_1, w_2, \dots, w_m) \in \mbq^m$, the diagram
\begin{equation*} 
\begin{CD}
\gal(\mbq(\zeta_\ell, {\bf{w}}^{1/\ell})/\mbq) @>>> (\mbz/\ell\mbz)^\times \ltimes (\mbz/\ell\mbz)^m \\
@V\text{res}VV @V\pi_jVV \\
\gal(\mbz(\zeta_\ell,w_j^{1/\ell})/\mbq) @>>> (\mbz/\ell\mbz)^\times \ltimes \mbz/\ell\mbz
\end{CD}
\end{equation*}
commutes, where $\pi_j({\bf{w}}) := w_j$.  Taking any prime $p$ as in the statement of the corollary, $p$ is unramified in $K$, and we fix a prime $\mf{P}$ of $K$ lying over $p$.   By the discussion preceding Lemma \ref{frobinkummerlemma}, for any multiplicative generator $z \in \mu_\ell \subseteq \mbf_p^\times$ we may find  $\ds i \in (\mbz/\ell\mbz)^\times$ for which 
\begin{equation} \label{defofpzi}
\mf{P} \cap \mc{O}_{\mbq(\zeta_\ell)} = \mf{p}_{z,i}.
\end{equation}
Let us fix an index $j \in \{ 1, 2, \dots, k \}$ and put $c := c_j$ and $d := d_j$.  Furthermore, denote by
\[
\mf{P}_{c} := \mf{P} \cap \mc{O}_{K_{c}} \quad \text{ and } \quad \mf{P}_{d} := \mf{P} \cap \mc{O}_{K_{d}}
\]
 the corresponding primes of $K_{c} := \mbq(\zeta_\ell, c^{1/\ell} )$ (resp. of $K_{d} := \mbq(\zeta_\ell, d^{1/\ell} )$) lying under $\mf{P}$.  Now if $c \pmod{p} \in (\mbf_p^\times)^\ell$, then necessarily $d \equiv c^{k_p} \pmod{p} \in (\mbf_p^\times)^\ell$, and by \eqref{splittingcompletelycriterionforc}, we have that $\frob_{\mf{P}_{c}} = (1,0) = \frob_{\mf{P}_{d}}$ (note that this covers the case $c \in (\mbq^\times)^\ell$).  In case $k_p \equiv 0 \pmod{\ell}$, we see by the same reasoning that $\frob_{\mf{P}_{d}} = (1,0)$, and one finds that in any of the above cases, the conclusion of the lemma holds, taking $\gl = 0$ in \eqref{defofmcC}.
 
We now turn to the case $c \pmod{p} \notin (\mbf_p^\times)^\ell$ and $\ell \nmid k_p$.  In this case, we put $z = z_{c}$ in \eqref{defofpzi}, possibly adjusting $i \pmod{\ell}$ appropriately.  Noting that $\theta_{d}$ only depends on $d$ modulo $p$, we find that $\theta_{d} = \theta_c^{k_p}$, and so $z_d = z_c^{k_p}$. Thus, by \eqref{defofmfPi} and \eqref{transformrule}, we find that
\[
\mf{P}_c = \mf{P}_{z_c,i} = \mf{P}_{z_c^{k_p},ik_p} = \mf{P}_{z_d,i k_p}.
\]
Applying Lemma \ref{frobinkummerlemma} and noting that the factor $k_p$ is independent of the index $j$, we have finished the proof.
\end{proof}

In particular, taking ${\bf{c}} = (n_1, n_2) \in \mbn^2$ and ${\bf{d}} = (f(n_1), f(n_2)) \in (\mbq^\times)^2$, we obtain the following corollary.  Recall that $\ds b_{f, {\bf{n}}} :=  \left| \prod_{i=1}^2 n_i \num(f(n_i)) \den(f(n_i)) \right|$.  
\begin{cor} \label{keycorinformneeded}
Suppose that $f : \mbq \longrightarrow \mbq$ is any function, $n_1, n_2 \in \mbn$, and 
$\ell$ is an odd prime number which doesn't divide $b_{f,{\bf{n}}}$. 
Then, with $\ds K = \mbq \left( \zeta_\ell, n_1^{1/\ell}, n_2^{1/\ell}, f(n_1)^{1/\ell}, f(n_2)^{1/\ell} \right)$, one has
\[
p \in S_f \, \text{ and } \, p \equiv 1 \pmod{\ell} \; \Longrightarrow \; \frob_p \subseteq \mc{C}_4 \; \text{ or } \; p \mid b_{f,{\bf{n}}},
\]
where $\mc{C}_4$ is defined by taking $k = 2$ in \eqref{defofmcC}.
\end{cor}

\subsection{Proof of Proposition \ref{secondsumboundprop}} \label{proofofsecondpropositionsection}

We now assume that $f$ is not a global power map, and define the constant $a_f > 0$ by
\[
a_f := 
\begin{cases}
\min \{ \max \{ \gd_{f,{\bf{n}}}, e^{b_{f,{\bf{n}}}} \} : \; {\bf{n}} \in \mc{N}_f \} & \text{ if } \mc{N}_f \neq \emptyset \\
1 & \text{ if } \mc{N}_f = \emptyset
\end{cases}
\]
(since $\gd_{f,{\bf{n}}} \geq 2$, we see that the minimum exists).  In case $|S_f| = \infty$, by Corollary \ref{lowerdegreeboundcor} we have that $\mc{N}_f \neq \emptyset$, so we may pick ${\bf{n}} \in \mc{N}_f$ for which 
\begin{equation*} 
a_f = \max \{ \gd_{f,{\bf{n}}}, e^{b_{f,{\bf{n}}}} \},
\end{equation*}
and then apply Theorem \ref{effectivechebotarevthm} with $K = \mbq(\zeta_\ell, n_1^{1/\ell}, n_2^{1/\ell}, f(n_1)^{1/\ell}, f(n_2)^{1/\ell} )$.  Note that in particular, provided $\ell > a_f$, by Corollary \ref{keycorinformneeded}, one has
\begin{equation} \label{kummercalc0}
\sum_{{\begin{substack} {p \in S_f(x) \\ p \equiv 1 \pmod{\ell}} \end{substack}}} 1 \leq \pi(x; K/\mbq, \mc{C}_4) + O(\nu(b_{f,{\bf{n}}})).
\end{equation}
Our assumption that
\begin{equation} \label{kummerboundforY}
Z \leq  \left( \frac{ \log x }{  (6 c_2 \log \log x )^2} \right)^{1/15}
\end{equation}
implies that, for $x$ large enough, one has $\sqrt{\log x / Z^5} \geq  6c_2 \cdot Z^{5}\log Z$.  By Corollary \ref{kummerdiscboundcor}, $\ell \in [Y,Z)$ and $a_f < Y$ guarantee that \eqref{xbigenough} holds in this case.   Thus, for $Y > a_f$ and $\ell \in [Y,Z)$, one has
\begin{equation} \label{kummercalc1}
\begin{split}
\pi(x; K/\mbq, \mc{C}_4 ) & = \frac{ | \mc{C}_4 |}{ | \gal(K/\mbq) |} \cdot \pi(x) + O \left( | \mc{C}_4 | \cdot x \cdot \exp\left( -c_1 \sqrt{\frac{\log x}{[K : \mbq ]}} \right) \right) \\
& \ll \frac{ | \mc{C}_4 |}{ | \gal(K/\mbq) |} \cdot \pi(x) + \ell^3 \cdot x \cdot \exp \left( -c_1 \sqrt{\frac{\log x}{Z^5}} \right).
\end{split}
\end{equation}
The following lemma bounds the first term above.
\begin{lemma} \label{chebotarevboundlemma}
Suppose that $f : \mbq \longrightarrow \mbq$ is any function, let $n_1, n_2 \in \mbn$, let $\ell$ be an odd prime, and let
$\ds K := \mbq \left( \zeta_\ell, n_1^{1/\ell}, n_2^{1/\ell}, f(n_1)^{1/\ell}, f(n_2)^{1/\ell} \right)$.  Suppose that
\begin{equation} \label{ellcubedassumption}
[\mbq \left( \zeta_\ell, n_1^{1/\ell}, n_2^{1/\ell}, f(n_1)^{1/\ell} \right) : \mbq(\zeta_\ell) ] = \ell^3.  
\end{equation}
Then one has
\[
\frac{ | \mc{C}_4 |}{ | \gal(K/\mbq) |} \; \leq \; \frac{2}{\ell(\ell - 1)},
\]
where $\mc{C}_4$ is defined by taking $k = 2$ in \eqref{defofmcC}.
\end{lemma}
\begin{proof}
By hypothesis, if $\ell$ is large enough then either
$
\gal(K/\mbq) \simeq (\mbz/\ell\mbz)^\times \ltimes (\mbz/\ell\mbz)^4
$
or
\[
\gal(K/\mbq) \simeq (\mbz/\ell\mbz)^\times \ltimes (\mbz/\ell\mbz \cdot {\bf{d}})^\perp,
\]
where ${\bf{d}} = (d_1, d_2, d_3, d_4) \in (\mbz/\ell\mbz)^4$ and
\begin{equation} \label{d4neq0}
d_4 \neq 0,
\end{equation}
which follows from the hypothesis \eqref{ellcubedassumption}.  If $\gal(K/\mbq) \simeq (\mbz/\ell\mbz)^\times \ltimes (\mbz/\ell\mbz)^4$, then directly from \eqref{defofmcC} one finds that
\[
| \mc{C}_4 | \leq \ell^3,
\]
and the conclusion of the lemma follows.  If on the other hand
\[
\gal(K/\mbq) \simeq (\mbz/\ell\mbz)^\times \ltimes (\mbz/\ell\mbz \cdot {\bf{d}})^\perp,
\]
then, writing ${\bf{d}} = ({\bf{d}}_1, {\bf{d}}_2)$ with ${\bf{d}}_i \in (\mbz/\ell\mbz)^2$, we have that
\[
\begin{split}
\mc{C}_4 & = \{ (1, {\bf{b}}, \gl {\bf{b}}) : \; ({\bf{b}},\gl) \in (\mbz/\ell\mbz)^3, \, {\bf{b}} \cdot {\bf{d}}_1 + \gl {\bf{b}} \cdot {\bf{d}}_2 = 0 \}  \\
& = \{ (1, {\bf{b}}, \gl {\bf{b}}) : \; ({\bf{b}},\gl) \in (\mbz/\ell\mbz)^3, \, {\bf{b}} \cdot ( {\bf{d}}_1 + \gl {\bf{d}}_2) = 0 \}.
\end{split}
\]
Consider the equation
\begin{equation} \label{bdefiningequation}
{\bf{b}} \cdot ( {\bf{d}}_1 + \gl {\bf{d}}_2) = 0.
\end{equation}
By \eqref{d4neq0} we see that ${\bf{d}}_2 \neq {\bf{0}} \in (\mbz/\ell\mbz)^2$, and so $ {\bf{d}}_1 + \gl {\bf{d}}_2 = {\bf{0}}$ for at most one $\gl \in \mbz/\ell\mbz$.  For such a $\gl$, one counts $\ell^2$ solutions ${\bf{b}} \in (\mbz/\ell\mbz)^2$ to the equation \eqref{bdefiningequation}, while for each of  the other $\ell - 1$ values of $\gl$ one counts $\ell$ solutions.  Thus, one has
\[
| \mc{C}_4 | \leq \ell(2\ell - 1),
\]
and the conclusion of the lemma follows in this case as well.
\end{proof}
Inserting the result of Lemma \ref{chebotarevboundlemma} into \eqref{kummercalc1} and using \eqref{kummercalc0} we obtain, that, for $\ell \in [Y,Z)$ we have
\[
\sum_{{\begin{substack} {p \in S_f(x) \\ p \equiv 1 \pmod{\ell}} \end{substack}}} 1 \; \ll \; \frac{1}{\ell^2} \cdot \pi(x) + \ell^3 \cdot x \cdot \exp \left( -c_1 \sqrt{\frac{\log x}{Z^5}} \right) + \nu(b_{f,{\bf{n}}}),
\]
provided \eqref{kummerboundforY} holds.  Summing over primes $\ell \in [Y,Z)$, we obtain
\[
\sum_{Y \leq \ell < Z} \sum_{{\begin{substack} {p \in S_f(x) \\ p \equiv 1 \pmod{\ell}} \end{substack}}} 1 \; \ll \; \frac{1}{Y \log Y} \cdot \pi(x) + \frac{Z^4}{\log Z} \cdot x \cdot \exp \left( -c_1 \sqrt{\frac{\log x}{Z^5}} \right) + \nu(b_{f,{\bf{n}}}), 
\]
By virtue of the bounds \eqref{kummerboundforY} and \eqref{expboundpoweroflog}, we see that the second remainder term satisfies
\[
\frac{Z^4}{\log Z} \cdot x \cdot \exp \left( -c_1 \sqrt{\frac{\log x}{Z^5}} \right) \; \ll_A \; \frac{x}{(\log x)^A}
\]
for any $A > 0$, and since $Y < Z$, this observation 
finishes the proof of Proposition \ref{secondsumboundprop}.

\begin{remark}
The hypothesis in Lemma \ref{chebotarevboundlemma} that $f$ not be a global power map is critical.  Indeed, if $f(\ga) = \ga^k$ for all $\ga \in \mbq$, then (e.g. provided $n_2$ is multiplicatively independent from $n_1$) under \eqref{newformofgal} one has
\[
\gal(K/\mbq) = (\mbz/\ell\mbz)^\times \ltimes \{ ({\bf{b}}, e{\bf{b}}) : \; {\bf{b}} \in (\mbz/\ell\mbz)^2 \}.
\]
In particular, one finds that
\[
\frac{ | \mc{C}_4 | }{| \gal(K/\mbq) | } = 
\frac{ | \{ (1, {\bf{b}}, e {\bf{b}}) : \; {\bf{b}} \in (\mbz/\ell\mbz)^2 \} |}{| (\mbz/\ell\mbz)^\times \ltimes \{ ( {\bf{b}}, e{\bf{b}}) : \; {\bf{b}} \in (\mbz/\ell\mbz)^2 \} | } = \frac{1}{\ell - 1},
\] 
and our method of proof fails for this case (as it should).
 \end{remark}
 
 \section{Acknowledgments}
 
This paper was motivated by a question posed by C. Khare in connection with compatible systems of one-dimensional Galois representations.  I thank Professor Khare for sharing this interesting question, and also Professor R. Khan, who originally communicated it to me.  Some of the research leading to this paper was done during a research stay at the Universit\"{a}t G\"{o}ttingen, and I would like to thank the university for providing a stimulating environment in which to work.  Finally, I would also like to thank Professor S. Basarab for stimulating discussions on this topic and for helpful comments.

\end{document}